%% file: main.tex
\documentclass[10pt]{amsart}

\usepackage[T1]{fontenc}
\usepackage{lmodern}
\usepackage{microtype}

\usepackage{amsmath}
\usepackage{amssymb}
\usepackage{amsthm}
\usepackage{amsxtra}
\usepackage{mathtools}
\usepackage{mathrsfs}
\numberwithin{equation}{section}

\usepackage{array}
\usepackage{booktabs}
\usepackage{longtable}
\usepackage{enumitem}
\setlist{itemsep=4pt}

\usepackage{xcolor}
\usepackage[margin=1.15in]{geometry}

\newcommand{\Disc}{\operatorname{disc}}
\newcommand{\GQ}{\Gal_{\Q}}

\definecolor{webcolor}{rgb}{0.8,0,0.2}
\definecolor{webbrown}{rgb}{.6,0,0}

\usepackage{hyperref}
\hypersetup{
  colorlinks=true,
  urlcolor=blue,
  citecolor=blue,
  linkcolor=blue,
  pdftitle={The groups SL2(F13) and SL2(F19) are Galois over Q},
  pdfauthor={Emir Eray Karabiyik},
  pdfsubject={Inverse Galois realizations of special linear groups},
  pdfkeywords={inverse Galois problem, special linear group, central embedding problem, Hurwitz space, Shimura curve}
}
\usepackage[alphabetic,backrefs,lite]{amsrefs}

\newtheorem{theorem}{Theorem}[section]
\newtheorem{proposition}[theorem]{Proposition}
\newtheorem{lemma}[theorem]{Lemma}
\newtheorem{corollary}[theorem]{Corollary}

\theoremstyle{definition}
\newtheorem{example}[theorem]{Example}

\theoremstyle{remark}
\newtheorem{remark}[theorem]{Remark}

\newcommand{\PP}{\mathbb P}
\newcommand{\QQ}{\mathbb Q}

\newcommand{\F}{\mathbb{F}}

\newcommand{\Q}{\mathbb{Q}}

\newcommand{\Z}{\mathbb{Z}}

\newcommand{\SL}{\operatorname{SL}}
\newcommand{\PGL}{\operatorname{PGL}}
\newcommand{\PSL}{\operatorname{PSL}}

\newcommand{\Gal}{\operatorname{Gal}}

\def\bbar#1{\setbox0=\hbox{$#1$}\dimen0=.2\ht0
  \kern\dimen0 \overline{\kern-\dimen0 #1}}

\title{The groups
  $\SL_2(\F_{13})$ and $\SL_2(\F_{19})$ are Galois over $\Q$}
\author{Emir Eray Karabiyik}
\address{Department of Mathematics, Cornell University,
  Ithaca, NY 14853, USA}
\email{ek693@cornell.edu}
\date{September 2, 2026}


\usepackage[euler-digits,small]{eulervm}

\begin{document}

\begin{abstract}
In this paper, we show that $\SL_2(\F_{13})$ and $\SL_2(\F_{19})$ are Galois groups of
totally real extensions of $\Q$. For each of these primes $\ell$, we find a polynomial of degree $\ell+1$ whose splitting field has Galois group $\PSL_2(\F_\ell)$. These fields satisfy B\"oge's criterion and the associated central embedding problem has a proper solution with Galois group $\SL_2(\F_\ell)$. One can choose the resulting fields totally real via a quadratic twist. The degree $14$ polynomial is found through a genus one Hurwitz family of degree 14 covers. The degree 20 polynomial is found using Yang's modular equations for the Shimura curve $X_6^*(1)$.

\end{abstract}

\maketitle

\section{Introduction}\label{sec:introduction}

\subsection{The inverse Galois problem for \texorpdfstring{$\SL_2(\F_\ell)$}
  {SL2(Fl)}}\label{subsec:intro-statement}
The inverse Galois problem asks whether every finite group occurs as the
Galois group over $\Q$. Zywina~\cite{Zywina2015} proved that the projective groups $\PSL_2(\F_\ell)$ are Galois over $\QQ$ for all primes $\ell\geq 5$. However, for the groups $\SL_2(\F_\ell)$, the answer is known only for $\ell=5,7,11$. One reason the inverse Galois problem is harder for
  $\SL_2(\F_\ell)$ is parity.  Two-dimensional representations attached to
  holomorphic modular eigenforms are odd, whereas a representation with
  image in $\SL_2(\F_\ell)$ has determinant one and is even.
 
A $\PSL_2(\F_\ell)$ extension $L/\Q$
gives a surjection $\varphi\colon\GQ\to\PSL_2(\F_\ell)$, and lifting it through the exact sequence
\begin{equation}\label{eq:central-extension}
 1\longrightarrow\{\pm I\}\longrightarrow\SL_2(\F_\ell)
 \longrightarrow\PSL_2(\F_\ell)\longrightarrow1
\end{equation}
is a \emph{central embedding problem} whose obstruction lives in
$H^2(\GQ,\{\pm1\})$. That obstruction depends on the arithmetic of the
particular field $L$ and not on the abstract group $\PSL_2(\F_\ell)$. Kl\"uners records $\PSL_2(\F_{11})$ extensions of $\Q$
that do not lift~\cite{Klueners2000} to $\SL_2(\F_{11})$. Hence, a realization of $\SL_2(\F_\ell)$
often needs a projective realization chosen for its local behavior.

The group $\SL_2(\F_5)$ was  realized  over
$\Q$ by Sonn~\cite{Sonn1980}, and a regular realization was given by Mestre as
a special case of his construction for $\widetilde A_n$~\cite{Mestre1990}. Plans produced a generic extension~\cite{Plans2007} for the same group. Mestre also
constructed a regular realization of $\SL_2(\F_7)$~\cite{Mestre1994}. In ~\cite{Klueners2000}, Kl\"uners found a degree-$24$ polynomial for $\SL_2(\F_{11})$, by a ray class field
computation carried out under the generalized Riemann hypothesis, and
obtained an unconditional regular realization by Mestre's
method.  We are not aware of a previously recorded
realization for $\ell>11$. 
Uttenthal observes that no explicit even irreducible mod $\ell$ (for $\ell \geq 13$) representation
is available~\cite{Uttenthal2025}*{Remark~20}. Our main result is the following.

\begin{theorem}\label{thm:main}
For $\ell=13$ and for $\ell=19$, there is a totally real Galois extension
$M_\ell/\Q$ with $\Gal(M_\ell/\Q)\cong\SL_2(\F_\ell)$.
\end{theorem}

\subsection{Projective stem fields}\label{subsec:intro-stems}
 Let $\ell$ be an odd prime, and let $\PSL_2(\F_\ell)$ act naturally on
  $\Omega_\ell=\PP^1(\F_\ell)$. This action is faithful and transitive and has
  degree $\ell+1$. If $L/\Q$ is Galois with group $\PSL_2(\F_\ell)$ and
  $B\leq\PSL_2(\F_\ell)$ is the stabilizer of a point of $\Omega_\ell$, we call
  the degree-$(\ell+1)$ field $L^B$ a \emph{projective stem field} of $L$.

  Equivalently, let $f_\ell\in\Q[x]$ be irreducible of degree $\ell+1$, and let
  $L_\ell$ be its splitting field. Suppose that an isomorphism
  \[
   \Gal(L_\ell/\Q)\simeq\PSL_2(\F_\ell)
  \]
  identifies the action on the roots of $f_\ell$ with the natural action on
  $\PP^1(\F_\ell)$. If $\alpha_\ell$ is a root of $f_\ell$, then
  \[
   \Q(\alpha_\ell)=L_\ell^B,
  \]
  where $B$ is the stabilizer of $\alpha_\ell$. Thus $\Q(\alpha_\ell)$ is a
  projective stem field. Kl\"uners's degree $12$ field for $\ell=11$ is an
  example~\cite{Klueners2000}. Our arithmetic task is to construct such stem
  fields and verify the local conditions appearing in B\"oge's criterion.

\begin{theorem}\label{thm:stems}
Let $\ell\in\{13,19\}$. For $\ell=13$, let $f_\ell\in\Z[x]$ be the degree $14$ polynomial displayed in Example~\textup{\ref{ex:f13}}, and for $\ell=19$, let $f_\ell\in\Z[x]$ be the degree $20$ polynomial in
Appendix~\textup{\ref{app:f19}}.  Then $f_\ell$ is irreducible,
all $\ell+1$ of its roots are real, and its splitting field $L_\ell$
satisfies
\[
 \Gal(L_\ell/\Q)\cong\PSL_2(\F_\ell).
\]
This isomorphism carries the action on the roots of $f_\ell$ to the action of
$\PSL_2(\F_\ell)$ on $\PP^1(\F_\ell)$.  Moreover, for every odd rational prime $q$
whose ramification index $e(q)$ in $L_\ell$ is even, the residue degree
$f(q)$ satisfies
\[
 f(q)\ \text{is odd}
 \qquad\Longleftrightarrow\qquad
 q\equiv1\pmod4 .
\]
\end{theorem}

The last condition is exactly the condition (2) of B\"oge's criterion  which we restate in Theorem \ref{thm:boege}. Assuming Theorem~\ref{thm:stems}, we prove Theorem~\ref{thm:main} in
\S\ref{subsec:proof-of-main}.

\begin{example}\label{ex:f13}
For $\ell=13$ one may take the following polynomial.
\input{polynomial13}By Proposition~\ref{prop:l13-projective} and Lemma~\ref{lem:l13-local}, its splitting field is a totally real $\PSL_2(\F_{13})$-extension of $\Q$,
ramified exactly at
\[
 2,\quad3,\quad13,\quad5119,\quad
 1854089821344611,\quad2086114950053051.
\]

\end{example}

The degree-$20$ polynomial $f_{19}$ has coefficients of up to $361$ decimal
digits, so we print it in Appendix~\ref{app:f19} rather than here.  By
Lemma~\ref{lem:l19-local} its splitting field is ramified exactly at $17$, at
$5605357159$, and at another prime of $100$ digits.

\begin{remark}\label{rem:not-claimed}
The polynomials $f_{13}$ and $f_{19}$ have Galois groups $\PSL_2(\F_{13})$ and $\PSL_2(\F_{19})$.  B\"oge's theorem only proves the existence of the fields $M_\ell$ of Theorem \ref{thm:main}. We made the lift for $\ell=13$ explicit using Crespo's
method~\cite{Crespo1997}, but the resulting polynomial has very large coefficients (on average order of $10^{180000}$) and we do not include it in this paper. Accompanying files contain the polynomial with Galois group $\SL_2(\F_13)$. 
\end{remark}

\subsection{Constructing the polynomials}\label{subsec:intro-constructions}
We explain how the two polynomials $f_\ell$ were found. They were obtained through different methods.

For $\ell=13$, we start from covers of $\PP^1$, branched over four points, whose branch
cycles lie in the conjugacy classes $2A,2A,2A,6A$ of $\PSL_2(\F_{13})$.  The associated absolute Hurwitz curve is a rational curve described by a degree $54$ Belyi map. The corresponding inner Hurwitz curve is also rational and parametrizes a family of degree $14$ covers with genus one source curves. Each rational point of the inner Hurwitz curve
  therefore determines one such cover defined over $\Q$. Specializing a chosen cover at a rational parameter for which all fourteen sheets are real produces $f_{13}$.  This construction is not used in the proof of Theorem \ref{thm:stems}. The Galois group of
$f_{13}$ is certified independently in two ways. First, we certify it by a squarefree
triple-sum resolvent of degree $364$ that splits as $182+182$, and secondly, a Magma script confirms our proof. We describe
the construction of $f_{13}$ in Appendix~\ref{app:f13-construction}.

For $\ell=19$, we begin with Yang's \cite{YangModularEquations}
modular equation of level $19$ on the Shimura curve $X_6^*(1)$.  Its analytic branches are indexed by twenty level-$19$ Hecke
cosets. An explicit reduction of the elliptic generators of the
quaternion order proves that the monodromy of the geometric cover is $\PGL_2(\F_{19})$. We use the discriminant square class to cut out the $\PSL_2(\F_{19})$ subcover. Since Yang's cover has quite large coefficients, we first replace and identify it by an elliptic curve providing a degree $20$ cover. After a quadratic base change that kills the sign character, it has harmonic branch values, i.e., with cross ratio $-1$. The quadratic base change gives a cover with a genus 2 domain. We move to another point on the rational inner Hurwitz curve it lies on, and get a specific rational cover $F_*$. A totally real specialization produces the polynomial $f_{19}$. Let $G$ be the Galois group of $f_{19}$. We obtain an upper bound of $G$ using the arithmetic monodromy of $F_*$. Factorization modulo small primes provide a lower bound, and we prove the equality $G=\PSL_2(\F_{19})$. An independent Magma calculation verifies the same conclusion.

\begin{remark}\label{rem:regular}
Both of the families constructed above are branched over four points of
$\PP^1$, and in each case we exhibit a specialization for which the central
embedding problem is solvable.  This is the setting in which Kl\"uners
upgraded his $\ell=11$ example to an unconditional \emph{regular}
realization, using Mestre's theorem on elements of $\mathrm{Br}_2(k(x))$ with four poles \cite{Mestre1994b}. The similarity with Mestre’s construction suggests that his argument may also apply to these two families, but we have not checked this.
\end{remark}

\subsection{Overview}\label{subsec:intro-outline}
In \S\ref{sec:common}, we recall B\"oge's criterion, prove the necessary lemmas concerning the local invariants among others. Using them, we deduce
Theorem~\ref{thm:main} from Theorem~\ref{thm:stems}.  In \S\ref{sec:l13}, we
prove Theorem~\ref{thm:stems} for $\ell=13$. We first prove that the Galois group of the splitting field of $f_{13}$ is $\PSL_2(\F_{13})$, which can also be checked using Magma. Afterwards we compute the local invariants of the field and invoke B\"oge's criterion.  The Hurwitz construction behind $f_{13}$ plays
no part in that proof and is deferred to
Appendix~\ref{app:f13-construction}.  In \S\ref{sec:l19}, we
prove Theorem~\ref{thm:stems} for $\ell=19$. The construction of the polynomial defining the projective stem field comes from Yang's~\cite{YangModularEquations}  modular equations and we include the construction in the same section.  In \S\ref{sec:consequence}, we present a corollary using Ramakrishna's deformation theorem \cite{Ramakrishna2002}. In \S\ref{sec:computations}, we present the necessary scripts and talk about their reproducibility.

\subsection{Notation}\label{subsec:notation}
Throughout, $\ell\in\{13,19\}$ and $q$ denotes a rational prime. In the cases $\ell$ denotes an arbitrary prime, we will explicitly state so. We write $\GQ=\Gal(\overline\Q/\Q)$, $\Omega_\ell=\PP^1(\F_\ell)$, and $B$ for a point stabilizer in $\PSL_2(\F_\ell)$,
that is, the image of a Borel subgroup.  The fields
\[
 K_\ell\subset L_\ell\subset M_\ell
\]
are, respectively, the projective stem field $\Q[x]/(f_\ell)$, its Galois closure, and the $\SL_2(\F_\ell)$ extension produced by Theorem~\ref{thm:main}.

For a rational prime $q$ and a Galois extension $L/\Q$, all primes of $L$
above $q$ have the same ramification index and residue degree, and we denote them by $e(q)$ and $f(q)$.  We reserve the subscripted symbols $f_{13}$ and
$f_{19}$ for the two polynomials of Theorem~\ref{thm:stems}. In a non-Galois field the
factorization of $q$ is recorded as a list of pairs $(e_i,f_i)$, with
multiplicities written as exponents.

Conjugacy classes of a finite group are named as in the
\textsc{Atlas}~\cite{Atlas1985}, and a transitive group of degree $n$ is
named $n T k$ as in the library of Hulpke~\cite{TransGrp}. We write a cycle type multiplicatively, i.e., $1^2 2^6$ is a permutation of $14$ points with two
fixed points and six transpositions.

\subsection*{Acknowledgements}

I would like to thank Ravi Ramakrishna for introducing me to this problem many years ago. I would also like to thank David Zywina for inspiring me regarding Inverse Galois Problem. 

The author used ChatGPT and Claude for exploratory literature searches, assistance with computer algebra code, and editorial feedback. The text in this paper is fully written by the author. The author assumes full responsibility for the contents of this paper. 
The computations in this paper were carried out with
SageMath~\cite{SageMath}, PARI/GP~\cite{PARI2}, GAP~\cite{GAP} and Magma~\cite{Magma}. The accompanying files can be found in 
\begin{center}
    \href{https://github.com/eekarabiyik/SL2F13}{https://github.com/eekarabiyik/SL2F13}.
\end{center}

\section{From a projective stem to \texorpdfstring{$\SL_2(\F_\ell)$}
  {SL2(Fl)}}\label{sec:common}

\subsection{The embedding problem}\label{subsec:embedding}
Let $L/\Q$ be Galois with group $\PSL_2(\F_\ell)$ and let
$\varphi\colon\GQ\twoheadrightarrow \PSL_2(\F_\ell)$ be the associated surjection.  A
\emph{solution} of the embedding problem attached to $\varphi$ and the exact sequence
\eqref{eq:central-extension} is a continuous lift of $\varphi$ to $\SL_2(\F_\ell)$, equivalently a continuous homomorphism
$\rho\colon\GQ\to\SL_2(\F_\ell)$ whose composition with the projection
$\SL_2(\F_\ell)\to \PSL_2(\F_\ell)$ is $\varphi$. It is called \emph{proper} if it is surjective, equivalently if the fixed field of $\ker(\rho)$ has Galois group $\SL_2(\F_\ell)$ over $\Q$.

The group $\SL_2(\F_\ell)$ is perfect so the extension \eqref{eq:central-extension} is nonsplit for every
$\ell\geq5$. The congruence
$\ell\equiv3,5\pmod8$ that appears below controls a different kind of splitting.
Assume $\ell\geq 11$. The natural action of $\PSL_2(\F_\ell)$ gives an embedding of $\PSL_2(\F_\ell)$ into $A_{\ell+1}$. Let $\widetilde A_{\ell+1}$ denote the unique nontrivial double cover of $A_{\ell+1}$, which is also called the \emph{Schur cover}. Restricting this double cover to $\PSL_2(\F_\ell)$ produces a central extension with kernel $\{\pm1\}$.  When $\ell\equiv3,5\pmod8$, the resulting extension is $\SL_2(\F_\ell)$, and otherwise it is the split extension $\{\pm I\}\times \PSL_2(\F_\ell)$.
  In the first case, the obstruction to the embedding problem is computed by Serre's trace-form formula \cite{Serre1984}. B\"oge turns this obstruction into the
arithmetic local criterion used below. 

\subsection{B\"oge's criterion}\label{subsec:boege}

\begin{theorem}[B\"oge]\label{thm:boege}
Let $\ell$ be a prime with $\ell\equiv3$ or $5\pmod 8$ and let $L/\Q$ be a
Galois extension with $\Gal(L/\Q)\cong\PSL_2(\F_\ell)$.  The embedding
problem attached to $L$ and \eqref{eq:central-extension} has a proper
solution if and only if
\begin{enumerate}
 \item $L$ is totally real, and
 \item for every odd prime $q$ with $e(q)$ even,
 \[
   f(q)\ \text{is odd}
   \quad\Longleftrightarrow\quad
   q\equiv1\pmod4 .
 \]
\end{enumerate}
\end{theorem}

This is an English restatement of Satz~1 of \cite{Boege1990}*{p.~153} which Kl\"uners also states as \cite{Klueners2000}*{Theorem~1.1}. There is no tameness hypothesis here. In particular, a wildly
ramified prime is allowed, and one occurs for $\ell=13$. There is also no condition at $2$.  Once the archimedean and the odd
local invariants vanish, Brauer--Hasse--Noether reciprocity forces the
dyadic invariant to vanish as well.

\subsection{Local invariants from a projective stem field}\label{subsec:orbit-dictionary}
  Condition~(2) of Theorem~\ref{thm:boege} is stated in terms of the
  ramification index and residue degree in the Galois extension $L$. However, direct computations in $L$ are impractical, so instead we factor the prime in the degree $(\ell+1)$ projective stem field $K=L^B$. The following standard lemma translates this factorization into the orbit data of the inertia and decomposition groups, from which the required indices can be recovered. Compare Theorem~3.1 of~\cite{BrumerKramer2012}. We include a proof in the
  precise form needed here.

\begin{lemma}\label{lem:orbit-dictionary}
Let $L/\Q$ be a finite Galois extension whose group $G$ acts faithfully and
transitively on a finite set $\Omega$, let $B$ be a point stabilizer, and put
$K=L^B$.  Fix a prime of $L$ above $q$ and let $I\trianglelefteq D\leq G$ be
its inertia and decomposition groups.  Then the primes of $K$ above $q$
correspond to the $D$-orbits on $\Omega$. An orbit $\mathcal O\ni\omega$ corresponds to a prime with
\begin{equation}\label{eq:orbit-dictionary}
 e(\mathcal O)=|I\omega|,
 \qquad
 f(\mathcal O)=\frac{|D\omega|}{|I\omega|} .
\end{equation}
\end{lemma}

\begin{proof}
Identify $\Omega$ with $G/B$.  The primes of $K=L^B$ above $q$ are indexed by
the double cosets $D\backslash G/B$, hence by the $D$-orbits on $G/B$.  If
$\omega=gB$ then the stabilizers of $\omega$ in $D$ and in $I$ are
$D\cap gBg^{-1}$ and $I\cap gBg^{-1}$, and local Galois theory gives
\[
 e(\mathcal O)=\frac{|I|}{|I\cap gBg^{-1}|}=|I\omega|,
 \qquad
 e(\mathcal O)f(\mathcal O)=\frac{|D|}{|D\cap gBg^{-1}|}=|D\omega|.
\]
This proves the assertion.  Since $I$ is normal in $D$, the
$I$-orbits inside a fixed $D$-orbit all have the same size, so
$e(\mathcal O)$ does not depend on the choice of $\omega\in\mathcal O$.
\end{proof}

In our applications, we use Lemma~\ref{lem:orbit-dictionary} in the converse direction.  Using GAP and PARI/GP we obtain the factorization type of $q$ in $K_\ell$ and
the different exponents. Afterwards, we do a finite search to list every pair $I\trianglelefteq D\leq \PSL_2(\F_\ell)$ producing that factorization type, and observe that in each case of
interest, there is only one possibility for $(e(q),f(q))$. For a tame prime, the search is easier since $I$ and $D/I$ are cyclic, and a Frobenius lift conjugates a generator of inertia to its $q$th power. For the
only wild prime that occurs, at $q=3$ in $K_{13}$, the search
retains a superset of the arithmetically possible pairs but then we observe that still the same numerical pairs $(e(q),f(q))$ are obtained for every candidate pair in the superset.

\begin{lemma}\label{lem:no-hidden-ramification}
With notation as in Lemma~\ref{lem:orbit-dictionary}, a prime that is
unramified in $K$ is unramified in $L$.
\end{lemma}

\begin{proof}
If $q$ were ramified in $L$ but not in $K$, then every $I$-orbit on $\Omega$
would be a singleton by \eqref{eq:orbit-dictionary}, so $I$ would act
trivially on $\Omega$.  As $G$ acts faithfully this forces $I=1$.

\end{proof}

\subsection{Totally real solutions}\label{subsec:twist}
The fixed field of a proper solution supplied by B\"oge's theorem does not have to be totally real. In our case, we can obtain one by a quadratic twist.
\begin{lemma}\label{lem:central-twist}
Let $\ell>3$ and let $L/\Q$ be a totally real Galois extension with group
$\PSL_2(\F_\ell)$.  If the embedding problem attached to $L$ and
\eqref{eq:central-extension} has a proper solution, then it has a proper
solution whose fixed field is totally real.
\end{lemma}

\begin{proof}
Note that $\SL_2(\F_\ell)$ is perfect. 
Let $\varphi\colon\GQ\twoheadrightarrow \PSL_2(\F_\ell)$ be the surjection attached to the Galois extension
$L$ and let $\rho$ be a proper solution to the embedding problem whose kernel cuts out a field with Galois group $\SL_2(\F_\ell)$. Let $c\in\GQ$ be a complex
conjugation.  Since $L$ is a totally real field, $\varphi(c)=1$ and $\rho(c)\in\{\pm I\}$.  If
$\rho(c)=I$, take $\chi$ to be the trivial character. Otherwise, let
$\chi\colon\GQ\to\{\pm I\}$ be the quadratic character of an imaginary
quadratic field.  In both cases put $\rho^\chi(g)=\chi(g)\rho(g)$. Since $\chi$ takes values in the center, the projection of $\rho^\chi$ is still $\varphi$ and $\rho^\chi$ is again a solution to the embedding problem. We have $\rho^\chi(c)=I$ by construction. 

It remains to show that $\rho^\chi$ is still surjective.  Its image maps onto
$\PSL_2(\F_\ell)$, so it is either all of $\SL_2(\F_\ell)$ or a complement of
$\{\pm I\}$, in which case its index is $2$.  The latter is impossible because a perfect
group has no subgroup of index $2$. Since all complex conjugations in $\GQ$ are conjugate, the fixed
field of $\ker\rho^\chi$ is totally real. 
\end{proof}

\subsection{Proof of Theorem~\ref{thm:main}}\label{subsec:proof-of-main}
Let $\ell\in\{13,19\}$ and let $L_\ell$ be the splitting field of $f_\ell$. By Theorem~\ref{thm:stems}, we have $\Gal(L_\ell/\Q)\cong\PSL_2(\F_\ell)$, and
$L_\ell$ is totally real. 
The last assertion of
Theorem~\ref{thm:stems} satisfies condition~(2) of B\"oge's criterion \ref{thm:boege} and since $13\equiv5\pmod8$ and
$19\equiv3\pmod8$, Theorem~\ref{thm:boege} applies and yields a proper solution $\rho$ of the embedding problem attached to $L_\ell$. The fixed field $M_\ell$ of $\ker\rho$ has $\Gal(M_\ell/\Q)\cong\SL_2(\F_\ell)$, and it can be chosen totally real by Lemma~\ref{lem:central-twist}.

\qed

\medskip

The next two sections prove Theorem~\ref{thm:stems}.

\section{The case \texorpdfstring{$\ell=13$}{l=13}}\label{sec:l13}

\subsection{The Galois group of \texorpdfstring{$f_{13}$}{f13}}
\label{subsec:l13-group}
The argument in this subsection uses nothing but the coefficients of
$f_{13}$. Proposition \ref{prop:l13-projective} proves the needed result using group theory, but it can be explicitly checked using Magma, cf. \S\ref{sec:computations}.

\begin{proposition}\label{prop:l13-projective}
The polynomial $f_{13}$ written in Example~\ref{ex:f13} is irreducible, all fourteen
of its roots are real, and its splitting field $L_{13}$ satisfies
\[
 \Gal(L_{13}/\Q)\cong\PSL_2(\F_{13})
\]
in the natural action of $\PSL_2(\F_{13})$ on $\PP^1(\F_{13})$.
\end{proposition}

\begin{proof}

Modulo $5$ the reduction of $f_{13}$ is a squarefree product of two
irreducible factors of degree $7$.  Modulo $23$ it factors as
\[
\begin{split}
 (x+8)(&x^{13}+9x^{12}+17x^{11}+19x^{10}+5x^9+15x^8+6x^7\\
 &+16x^6+19x^5+15x^4+12x^3+4x^2+21x+7)
\end{split}
\]
with the degree $13$ factor irreducible.  A nontrivial factor of $f_{13}$ over $\Q$ would have degree $7$ by the first reduction and degree $1$ or
$13$ by the second, so $f_{13}$ is irreducible and its Galois group $G$ is
transitive of degree $14$.

The discriminant of $f_{13}$ is a nonzero square, so $G\leq A_{14}$.  The
reduction modulo $23$ is squarefree, so $23\nmid\Disc(f_{13})$ and
so $G$ contains an element of cycle type
$1\cdot13$, that is, a $13$-cycle.  Among the $63$ transitive groups of
degree $14$ classified by Butler~\cite{Butler1993} exactly two are contained
in $A_{14}$ and contain a $13$-cycle, namely
\[
 14T30\cong\PSL_2(\F_{13})
 \qquad\text{and}\qquad
 14T62=A_{14} .
\]
We verified this by inspecting the entire library~\cite{TransGrp} in GAP.

Let $r_1,\ldots,r_{14}$ be the roots of $f_{13}$.  To each three-element set
$\{i,j,k\}$ attach the sum $r_i+r_j+r_k$, and put all these sums into one
polynomial
\[
 R_3(y)=\prod_{1\leq i<j<k\leq14}
       \bigl(y-(r_i+r_j+r_k)\bigr),
\]
which is the triple sum resolvent. We compute it exactly from the coefficients of $f_{13}$ with no numerical
approximation of any of the roots.

 Let $s_m=\sum_i r_i^m$. For $a\in\{1,2,3\}$, define
  \[
   P_a(z)=\sum_{m=0}^{364}\frac{a^ms_m}{m!}z^m.
  \]
  Then, modulo $z^{365}$,
  \[
   \sum_{i<j<k}\exp\bigl((r_i+r_j+r_k)z\bigr)
   =\frac{P_1(z)^3-3P_1(z)P_2(z)+2P_3(z)}6.
  \]
All equalities here are taken modulo $z^{365}$.  Newton identities compute
$s_1,\ldots,s_{364}$ directly from $f_{13}$.  For the sum,  $m!$ times the
coefficient of $z^m$ is
\[
 \sum_{i<j<k}(r_i+r_j+r_k)^m,
\]
the $m$-th power sum of the $364$ roots of $R_3$.  Applying Newton identities
a second time therefore reconstructs every coefficient of $R_3$.

 Each
coefficient of $R_3$ is a symmetric polynomial with integer coefficients in $r_1,\ldots,r_{14}$.  The fundamental theorem of symmetric polynomials
expresses it as an integer polynomial in the elementary symmetric functions
of the $r_i$ and consequently $R_3\in\Z[y]$.

We show, using Sage, that $R_3$ is squarefree and is the product
of two irreducible polynomials over $\Q$, each of degree $182$.
Thus
\[
 \{i,j,k\}\longmapsto r_i+r_j+r_k
\]
is a $G$-equivariant bijection from the three-subsets of the roots to the
roots of $R_3$.  Under this bijection, the irreducible factors of $R_3$
correspond exactly to the $G$-orbits.  Hence $G$ has two orbits, both of size
$182$, on the three-subsets.

The group $A_{14}$ is transitive on all $364$
three-subsets.  In contrast, the natural degree-$14$ action of $\PSL_2(\F_{13})$ has precisely two orbits of size $182$.  As a result, we conclude
that $G\cong\PSL_2(\F_{13})$ in its natural action. Sturm arithmetic in the accompanying
Sage file gives fourteen real roots~\cite{BasuPollackRoy2006}*{\S2.2.2}.

Additionally, an exact Magma script in the level $13$
code package proves irreducibility and, using \texttt{GaloisProof},
identifies the Galois group in its degree-$14$ action with
$\PSL_2(\F_{13})$. 
\end{proof}

\subsection{Ramification}\label{subsec:l13-local}
Write $K_{13}=\Q[x]/(f_{13})$ and set
\[
 p_1=1854089821344611,\qquad p_2=2086114950053051 ,
\]
both prime.  We have:
\begin{equation}\label{eq:l13-field-discriminant}
 \Disc(K_{13})=2^{22}\,3^{22}\,13^{6}\,5119^{6}\,p_1^{6}\,p_2^{6} .
\end{equation}

\begin{lemma}\label{lem:l13-local}
The primes ramified in $L_{13}$ are exactly $2,3,13,5119,p_1,p_2$, and at
each odd one the pair $(e(q),f(q))$ is as in Table~\textup{\ref{tab:l13-local}}.
\end{lemma}

\begin{table}[ht]
\centering
\begin{tabular}{@{}clcc@{}}
\toprule
$q$ & type of $q$ in $K_{13}$ & $(I,D)$ & $(e(q),f(q))$\\
\midrule
$2$ & $(1,2),(6,1)^2$ & --- & not needed\\
$3$ & $(2,1),(3,2),(6,1)$ & $(S_3,D_{12})$ & $(6,2)$\\
$13$ & $(1,1)^2,(2,3)^2$ & $(C_2,C_6)$ & $(2,3)$\\
$5119,\,p_1,\,p_2$ & $(1,2),(2,1)^2,(2,2)^2$
 & $(C_2,C_2\times C_2)$ & $(2,2)$\\
\bottomrule
\end{tabular}
\caption{Local data for $K_{13}$ and for its Galois closure $L_{13}$.}
\label{tab:l13-local}
\end{table}

\begin{proof}
We construct the maximal order of $K_{13}$ using PARI/GP, certify it with
\texttt{nfcertify}, and also confirm the discriminant \eqref{eq:l13-field-discriminant}. Factoring the six ramified primes in this maximal order gives the second column of Table~\ref{tab:l13-local}.

We now apply Lemma~\ref{lem:orbit-dictionary} with $G=\PSL_2(\F_{13})$ acting on fourteen points, and enumerate the pairs $I\trianglelefteq D\leq G$ which can realize each factorization type. At $q=3$ the enumeration retains a superset of the possible wild pairs. For each odd prime $q$ in Table \ref{tab:l13-local}, these enumerated sets contain only one conjugacy class of pairs $(I,D)$. In every case, we obtain the invariants appearing in Table \ref{tab:l13-local}.

Every prime ramified in $K_{13}$ ramifies in $L_{13}$, and by
Lemma~\ref{lem:no-hidden-ramification} there is no further ramification, proving the assertion.

\end{proof}

Only the odd primes with $e(q)$ even are considered for Theorem~\ref{thm:boege}, and
those are $3,13,5119,p_1$ and $p_2$.  For them,
\[
\begin{array}{c|ccccc}
 q & 3&13&5119&p_1&p_2\\
 \hline
 f(q) & 2&3&2&2&2\\
 q\bmod4 & 3&1&3&3&3
\end{array}
\]
so $f(q)$ is odd exactly when $q\equiv1\pmod 4$.  Lemma~\ref{lem:l13-local}
and Proposition~\ref{prop:l13-projective} together prove
Theorem~\ref{thm:stems} for $\ell=13$.

\section{The case \texorpdfstring{$\ell=19$}{l=19}}\label{sec:l19}

For $\ell=19$, we include the construction of the polynomial here, as it involves some Shimura curve results. We start with Yang's \cite{YangSchwarzian} cover, which comes from Hecke correspondences,  with geometric monodromy $\PGL_2(\F_{19})$. After a quadratic base change one obtains a cover with monodromy $\PSL_2(\F_{19})$. We transform this pullback cover into a form that allows better computations and obtain the polynomial $f_{19}$ by a precise search.

\subsection{Yang's Hecke cover}\label{subsec:l19-yang}
Let $\mathcal B=(-1,3)_\Q$ be the indefinite quaternion algebra of
discriminant $6$,
so that
\[
 I^2=-1,\qquad J^2=3,\qquad IJ=-JI ,
\]
and let
\[
 \mathcal O=\Z+\Z I+\Z J+\Z\tfrac{1+I+J+IJ}{2}
\]
be a maximal order.  The full Atkin--Lehner quotient $X_6^*(1)$ of the Shimura curve $X_6(1)$ has
genus $0$.  Yang~\cite{YangSchwarzian} fixes a Hauptmodul $T$ on $X_6^*(1)$ by requiring the
elliptic points of orders $6$, $2$ and $4$ to lie at $0$, $-540$ and
$\infty$. We use Yang's modular equation algorithm~\cite{YangModularEquations} to produce a level $19$ modular equation associated with this Hauptmodul. Regarding it as a polynomial in $z$ over $\Q(T)$ and dividing by its leading coefficient, we obtain a monic polynomial $P_0(T,z)\in\Q(T)[z]$ of degree $20$ in $z$. 
 More precisely, the level-$19$ Hecke double coset is the disjoint union of twenty right cosets. Let $\mathcal C_{19}$ denote the set of these cosets. On a simply connected open
set of the $T$-line, each $c\in\mathcal C_{19}$ determines an analytic branch $z_c$. We have 
\begin{equation}\label{eq:yang-defining-product}
 P_0(T,z)=\prod_{c\in\mathcal C_{19}}(z-z_c)\in\Q(T)[z].
\end{equation}
Thus, for a generic point $x\in X_6^*(1)$, the roots of
  $P_0(T(x),z)$ are the Hauptmodul values of the twenty points related to $x$ by the level-$19$ Hecke correspondence. Clearing denominators and removing the common content gives the primitive
integral model $P(T,z)\in\Z[T,z]$ of bidegree $(6,20)$. We use this integral model $P$ for the rest of the paper. Its exact coefficients can be found in the accompanying files. The following lemma helps us identify the monodromy of $P$.

\begin{lemma}\label{lem:l19-hecke-cosets}
Choose a splitting $\mathcal O\otimes\Z_{19}\simeq M_2(\Z_{19})$. It
gives natural equivariant bijections
\[
 \mathcal C_{19}\ \simeq\
 \PGL_2(\F_{19})/B(\F_{19})\ \simeq\ \PP^1(\F_{19}).
\] 
Under these bijections, analytic continuation along a loop around an elliptic
point permutes the roots $z_c$ through the image in $\PGL_2(\F_{19})$ of a
generator of the corresponding elliptic stabilizer.
\end{lemma}

\begin{proof}

Put $\mathcal O_{19}=\mathcal O\otimes\Z_{19}$, and let
  $\mathcal E_{19}\subset\mathcal O_{19}$ be the local Eichler order of level
  $19$. Define
  \[
   K=\mathcal O_{19}^{\times}/\Z_{19}^{\times},
   \qquad
   K_0(19)=\mathcal E_{19}^{\times}/\Z_{19}^{\times}.
  \]  
Reduction modulo $19$ maps $K$ onto
$\PGL_2(\F_{19})$, and $K_0(19)$ is the inverse image of the upper triangular
Borel subgroup $B(\F_{19})$.  The kernel of reduction is contained in
$K_0(19)$, so reduction induces a bijection
\[
 K/K_0(19)\ \simeq\
 \PGL_2(\F_{19})/B(\F_{19}).
\]
The cosets $\mathcal C_{19}$ indexing Yang's double-coset product are right
cosets, and hence form the set $K_0(19)\backslash K$.  We identify this set
with the left-coset space above by inversion:
\[
 K_0(19)g\longmapsto g^{-1}K_0(19).
\]
Under this identification, right multiplication by $k$ corresponds to left
multiplication by $k^{-1}$.  We use this left-coset convention below.
Let $L_0$ be the line spanned by $(1,0)$ in $\F_{19}^2$.  Its stabilizer in
$\PGL_2(\F_{19})$ is $B(\F_{19})$, so
\[
 gB(\F_{19})\longmapsto gL_0
\]
is a well-defined equivariant bijection from the displayed coset space to
the set of lines in $\F_{19}^2$, namely $\PP^1(\F_{19})$.  In particular,
the three sets in the lemma each have twenty elements.

Formula~\eqref{eq:yang-defining-product} attaches the factor $z-z_c$ to the
coset $c$.  A small positively oriented loop around an elliptic point acts on
the Hecke correspondence by permuting its cosets.  This action is induced by
a generator of the corresponding elliptic stabilizer, so analytic
continuation carries the factor $z-z_c$ to the factor indexed by the image of
$c$.
Under the chosen splitting and the inversion convention above, reduction
modulo $19$ turns this coset action into the projective action stated in the
lemma.
\end{proof}

To use the lemma, we need the splitting explicitly. Using Hensel's lemma, we get a
$b\in\Z_{19}$ with $b^2=-13$ and $b\equiv5\pmod {19}$. Take
\[
 I\longmapsto\begin{pmatrix}0&-1\\1&0\end{pmatrix},
 \qquad
 J\longmapsto\begin{pmatrix}4&b\\b&-4\end{pmatrix}
 \]
in $M_2(\Z_{19})$. These matrices satisfy the quaternion relations given at the beginning of the section. Modulo $19$ the second matrix is
$\left(\begin{smallmatrix}4&5\\5&-4\end{smallmatrix}\right)$, and the images
of the four displayed generators of $\mathcal O$ span $M_2(\F_{19})$.
Using Nakayama's lemma, we see that the induced map
$\mathcal O\otimes\Z_{19}\to M_2(\Z_{19})$ is bijective. Let $\zeta_2$, $\zeta_4$, and $\zeta_6$ denote generators of the elliptic stabilizers of orders $2$, $4$, and $6$ described in \cite{YangSchwarzian}*{\S4}. Substituting the displayed matrices for $I$ and $J$ gives, up to scalars,
\[
 \zeta_2=I(3+J),\qquad\zeta_4^{-1}=1-I,\qquad\zeta_6=(1-I)(3+J).
\]
The inverse at order $4$ gives the required product-one relation, since
\[
 \zeta_2\zeta_4^{-1}\zeta_6=12,
\]
and the scalar $12$ is trivial in $\PGL_2$. Reducing these elements modulo $19$ gives permutations of $\PP^1(\F_{19})$ with cycle types
\begin{equation}\label{eq:l19-folded-passport}
 2^{10},\qquad 4^5,\qquad 1^2 6^3.
\end{equation}
The group these three permutations generate has order
$19(19^2-1)=6840$.

\begin{proposition}\label{prop:l19-yang-pgl}
The geometric monodromy group of $P(T,z)$, in its natural degree $20$ action, is
$\PGL_2(\F_{19})$.
\end{proposition}

\begin{proof}
By Lemma~\ref{lem:l19-hecke-cosets} the monodromy of the three orbifold
generators is the explicit matrix action just described.  Those three
permutations generate a group of order $6840=|\PGL_2(\F_{19})|$ contained in
the image of $\PGL_2(\F_{19})$, so the two coincide.
\end{proof}

\subsection{The sign pullback and a harmonic model}\label{subsec:l19-sign}
The determinants of the three elliptic stabilizers have square classes
$(+,-,-)$, which are also the signs of the three permutations in
\eqref{eq:l19-folded-passport}.  Since the matrices generate
$\PGL_2(\F_{19})$, the kernel of the sign character on the monodromy is
exactly $\PSL_2(\F_{19})$.  The stored equation of $P(T,z)$ satisfies 
\[
 \Disc_z P(T,z)=(-95T)\,\Delta_0(T)^2,
 \qquad\Delta_0\in\Z[T],\quad\deg\Delta_0=106 ,
\]
for some $\Delta_0$. Since adjoining the square root of a discriminant kills the sign character, the base change
\begin{equation}\label{eq:l19-sign-basechange}
 T=-95s^2
\end{equation}
gives the sign double cover, and the normalization of the pullback has
geometric monodromy $\PSL_2(\F_{19})$ which is connected since $\PSL_2(\F_{19})$ acts transitively on $\PP^1(\F_{19})$. These arguments prove the following:

\begin{lemma}\label{lem:l19-sign-pullback}
Let $C\to\PP^1_s$ be the normalization of the pullback of $P(T,z)$ along
$T=-95s^2$.  Its geometric monodromy group is $\PSL_2(\F_{19})$ in the
natural action on $\PP^1(\F_{19})$.
\end{lemma}

 For the cover $C\to\PP^1_s$ of Lemma~\ref{lem:l19-sign-pullback}, the branch
  values $T=0,-540,\infty$ pull back under $T=-95s^2$ to
  \[
   s=0,\qquad s=\pm\sqrt{108/19},\qquad s=\infty.
  \]
The equation $P$, reconstructed from Yang's algorithm, is too large to use directly in the deformation calculation (by "deform" we mean varying the point on the Hurwitz curve to obtain a new cover preserving the cycle types), so we introduce a smaller exact model and identify its monodromy below. In a suitable ordering, the cross-ratio of the branch values of this pullback is
$-1$.  We call such a configuration \emph{harmonic}. The replacement model has as its source the
elliptic curve
\begin{equation}\label{eq:l19-harmonic-elliptic}
 E:\quad v^2=9747u^3-18945u^2+11385u-2187
\end{equation}
and covering map
\begin{equation}\label{eq:l19-harmonic-map}
 \beta=95\cdot19\rho^2\,u\,\frac{g^6}{h^4}\ \colon\ E\longrightarrow\PP^1_T,
 \qquad
 \rho=\frac{11250000}{116490258898219},
\end{equation}
Here $g,h\in\Q(E)$. Their exact coefficients are included in the supplementary computational files.

We have found this replacement model as follows. The deck transformation of the
quadratic base change induces the involution
\[
 \iota\colon C\longrightarrow C,\qquad (s,z)\longmapsto(-s,z)
\]
on the normalized sign pullback.  Over finite fields, the quotient by
$\iota$ has genus one, and its reductions identify isogeny class $114\mathrm a$
as a candidate for $E$.  The required ramification partitions translate into
divisor equations for the covering function on this elliptic curve. We first find a solution to these equations over $\F_7$. We then lift this solution using Hensel's lemma and recover rational coefficients.  Direct substitution verifies that these coefficients satisfy the defining identity. The calculation over $\F_7$ is only how we found the
  model.  The proof uses the exact identity and the comparison below.

The ramification partitions above the three branch values of $\beta$ are
\[
 \beta^{-1}(0):1^2 6^3,
 \qquad
 \beta^{-1}(-540):2^{10},
 \qquad
 \beta^{-1}(\infty):4^5 .
\]
Three matching fiber types do not make our replacement cover \eqref{eq:l19-harmonic-map} the same
cover as Yang's nor determine its monodromy. By comparing the two
reductions modulo $23$, the next lemma proves that the harmonic model
nevertheless has Yang's geometric monodromy group.

\begin{lemma}\label{lem:l19-harmonic-is-yang}
Let $\pi_Y\colon X_Y\to\PP^1_T$ be Yang's degree-$20$ cover obtained by
normalizing $P(T,z)=0$.  The covers $\pi_Y$ and
$\beta\colon E\to\PP^1_T$ have good tame reduction at $23$. Their
special fibers are isomorphic as degree $20$ covers of
$\PP^1_{\F_{23}}$.  Consequently, the harmonic cover $\beta$ has geometric
monodromy $\PGL_2(\F_{19})$ in its natural degree $20$ action.
\end{lemma}

\begin{proof}
Let $X_Y\to\PP^1_T$ denote Yang's degree-$20$ cover, obtained by normalizing
$P(T,z)=0$, and let $\beta\colon E\to\PP^1_T$ denote the harmonic cover given above.
We compare these two covers modulo $23$.  Put $k=\F_{23}$.  The cubic in
\eqref{eq:l19-harmonic-elliptic} remains squarefree over $k$ and the map
$\beta$ retains degree $20$. The displayed ramification partitions and
the three branch values all remain distinct.  Since its ramification indices
are $2$, $4$, and $6$, all prime to $23$, the harmonic cover has good tame
reduction.

The reduction $\overline P(T,z)$ is irreducible of degree $20$ over $k(T)$.
Substituting $T=\beta$ and taking the norm
from $k(E)$ to $k(u)$ produces factors of degrees
$2$ and $38$,
and the greatest common divisor of the degree-two factor with
$\overline P(\beta,z)$ is
linear over $k(E)$.  This exhibits an embedding
$k(T)[z]/(\overline P)\hookrightarrow k(E)$ of $k(T)$-algebras. Both $k(T)$ algebras have degree
$20$ over $k(T)$, so the embedding gives an isomorphism of covers.
Thus the normalized reduction of Yang's cover is the same smooth tame cover
as the reduction of the harmonic model.

Now compare characteristic $0$ with characteristic $23$. A degree-$20$
monodromy group has order dividing $20!$, and $23\nmid20!$, so the monodromy
is prime to $23$ and is therefore unchanged by specialization
\cite{SGA1}*{Expos\'e~XIII, Cor.~2.9 and \S2.10}. Both characteristic zero
covers thus have the monodromy of the common special fiber, which is
$\PGL_2(\F_{19})$ by Proposition~\ref{prop:l19-yang-pgl}.
\end{proof}

We have seen that Yang's cover and $\beta$ have the same monodromy, but we also need to express the quadratic base change on $E$ to eliminate the sign character. Form the fiber product
\[
 E\times_{\PP^1_T}\PP^1_s,
\]
using $\beta\colon E\to\PP^1_T$ and $T=-95s^2$, and let $C_{\mathrm h}$
be its normalization. The induced map
$\pi_{\mathrm h}\colon C_{\mathrm h}\to\PP^1_s$ is the harmonic sign
pullback of the cover $\beta$. It has degree $20$, and its function field is
\[
 \Q(C_{\mathrm h})=\Q(E)(s),\qquad s^2=-\beta/95.
\]
Equation~\eqref{eq:l19-harmonic-map} gives the identity
\[
 -\frac{\beta}{95}
   =\Bigl(\rho\frac{g^3}{h^2}\Bigr)^2(-19u).
\]
Consequently the element
\[
 W=\frac{s h^2}{\rho g^3}\in\Q(C_{\mathrm h})
\]
satisfies
\[
 W^2=-19u,
 \qquad
 s h^2=\rho Wg^3.
\]
Conversely, the second identity recovers $s$ from $W$, so we have an equality of fields
$\Q(C_{\mathrm h})=\Q(E)(W)$. As a result, $W$ describes the specified sign pullback, rather than a different quadratic extension of $\Q(E)$.

We next eliminate the elliptic coordinate $v$.  Write
$E\colon v^2=f_3(u)$ and decompose the two explicit functions $g$ and $h$ of \ref{eq:l19-harmonic-map} as
\[
 h^2=h_0(u)+v h_1(u),
 \qquad
 g^3=g_0(u)+v g_1(u),
\]
with $h_i,g_i\in\Q(u)$.  After substituting $u=-W^2/19$, the relation
$s h^2-\rho Wg^3=0$ becomes
\[
 \Phi_0(s,W)+v\Phi_1(s,W)=0,
\]
where
\[
 \begin{aligned}
 \Phi_0(s,W)
   &=s h_0(-W^2/19)-\rho Wg_0(-W^2/19),\\
 \Phi_1(s,W)
   &=s h_1(-W^2/19)-\rho Wg_1(-W^2/19).
 \end{aligned}
\]
Both functions $\Phi_0$ and $\Phi_1$ are explicit and linear in $s$. We multiply them by a
common denominator to make them polynomials in $s$ and $W$.  The involution of the quadratic extension $\Q(W,v)/\Q(W)$ sends $v$ to $-v$.  Its norm
eliminates $v$:
\[
 \operatorname{Norm}(\Phi_0+v\Phi_1)
   =(\Phi_0+v\Phi_1)(\Phi_0-v\Phi_1)
   =\Phi_0^2-f_3(-W^2/19)\Phi_1^2.
\]
On the dense open set where $\Phi_1\ne0$, the equation $G(s,W)=0$ recovers $v=-\frac{\Phi_0(s,W)}{\Phi_1(s,W)}$. The norm identity then gives $v^2=f_3(-W^2/19)$, so the recovered $v$satisfies the original relation. Hence the primitive factor $G(s,W)$ and
$C_{\mathrm h}$ have the same function field.
  After substitution, we clear the denominators of $g$ and $h$. These denominators vanish at $u=1$, equivalently at $W^2+19=0$, and clearing them introduces the factor $(W^2+19)^2$. This factor is supported entirely on the denominator locus and is therefore absent from the original rational equation. Canceling it preserves the function field of the pullback. Dividing the remaining polynomial by its content and fixing its overall sign
  then gives the primitive integral equation 
\begin{equation}\label{eq:l19-harmonic-eliminant}
 G(s,W)=G_2(W)s^2+G_1(W)s+G_0(W)=0.
\end{equation}
It has degree $2$ in $s$ because the norm multiplies two expressions linear
in $s$, and the exact calculation gives degree $20$ in $W$. The coefficients of each $G_i$ can be found in the accompanying files. Accordingly,
projection to the $W$-line is a double cover, while projection to the
$s$-line is the required degree-$20$ map $\pi_{\mathrm h}$.

The discriminant of \eqref{eq:l19-harmonic-eliminant} as a quadratic in
$s$ is
\[
 \Delta_G(W)=G_1(W)^2-4G_0(W)G_2(W)
            =H_6(W)R_{17}(W)^2,
\]
where $H_6$ is squarefree of degree $6$ and $R_{17}$ has degree $17$. Completing the square in $s$ and setting
\[
 V=\frac{2G_2(W)s+G_1(W)}{R_{17}(W)}
\]
identifies the function field of the curve $G(s,W)=0$ with the function field of
\[
 V^2=H_6(W).
\]
Thus we can remove $R_{17}^2$ in a birational change of coordinates. The zeros contributed by $R_{17}^2$ do not give branch points after normalization.  The six simple roots of $H_6$ are the branch points of the
double cover of the $W$-line and since $\deg H_6=6$ is even, we see that infinity contributes no further branch point. Finally, Riemann--Hurwitz formula gives
\[
 2g(C_{\mathrm h})-2=2(-2)+6=2,
\]
so $C_{\mathrm h}$ has genus $2$.  Since the base change is the sign
pullback, the geometric monodromy of $\pi_{\mathrm h}$ is
$\PSL_2(\F_{19})$. The accompanying computation reconstructs $G(s,W)$ exactly from $g$ and $h$ and verifies the resulting identity coefficient by coefficient.

\subsection{A rational genus-two cover}\label{subsec:l19-deformation}
The harmonic cover is defined over $\Q$ with a special branch configuration, but specializing it does not give a totally real polynomial.  We
therefore move along the family of covers with the same local data until we
reach a point where it does, keeping track of rationality.

Consider the Nielsen class $(2A,2A,2A,3A)$ in $\PSL_2(\F_{19})$, and let
  $\mathcal H^{\mathrm{in}}$ and $\mathcal H^{\mathrm{abs}}$ denote its inner
  and absolute Hurwitz curves. Riemann--Hurwitz applied to the passport of the
  exact degree-$63$ Belyi map shows that $\mathcal H^{\mathrm{abs}}$ has genus
  zero. We normalize a coordinate $\upsilon$ on this curve by sending its
  unique cusps of widths $1$, $3$, and $2$ to $0$, $1$, and $\infty$,
  respectively.

  Since $N_{S_{20}}(\PSL_2(\F_{19}))/\PSL_2(\F_{19})=\PGL_2(\F_{19})/\PSL_2(\F_{19})$ has order $2$, the map
  $\mathcal H^{\mathrm{in}}\to\mathcal H^{\mathrm{abs}}$ has degree $2$.
  The reduced braid action shows that this map is branched precisely above
  $\upsilon=0$ and $\upsilon=4$. Consequently, its function field has the
  form
  \[
   \Q(\mathcal H^{\mathrm{in}})
   =\Q(\upsilon)\bigl(\sqrt{d\,\upsilon(\upsilon-4)}\bigr),
   \qquad d\in\Q^\times/\Q^{\times2}.
  \]

  The fiber above $\upsilon=1$ consists geometrically of two inner cusps, which together form a degree-$2$ closed point over $\Q$. At either cusp, the
  degree-$20$ cover degenerates into two irreducible components. The induced three-point covers on these components have monodromy groups $A_4$ and $S_3$, respectively, and the components meet at a node with inertia group $C_3$.
  Complex conjugation exchanges the two cusps. The branch cycle lemma
  \cite{BaileyFried2002}*{\S3.2.2} therefore identifies the residue field of this degree-$2$ closed point with $\Q(\zeta_3)=\Q(\sqrt{-3})$.
  On the other hand, substituting $\upsilon=1$ gives the residue field $\Q(\sqrt{-3d})$. Hence $d$ is a square. After rescaling the second coordinate to absorb $d$, the inner Hurwitz curve has the affine model
  \[
   \omega^2=\upsilon(\upsilon-4).
  \]
  This conic admits the rational parametrization
\begin{equation}\label{eq:l19-inner-line}
 \upsilon=2+t+t^{-1},\qquad \omega=t-t^{-1} .
\end{equation}
The harmonic point is given by $t=1$.  We move to the point
\begin{equation}\label{eq:l19-inner-target}
 t_*=-\frac{45}{13},
 \qquad
 (\upsilon_*,\omega_*)=
 \Bigl(-\frac{1024}{585},-\frac{1856}{585}\Bigr),
\end{equation}
chosen for two reasons.  First, it lies in an interval of the real inner
line over which all four branch points stay real and complex conjugation acts
trivially on the Nielsen class. Hence the covers above $t_*$ can have totally
real fibers.  Second, it is congruent to the harmonic point modulo $29$,
which is what makes the Hensel lift below possible.

We first express the three finite branch values of the cover in terms of the
  coordinate $t$ on the inner Hurwitz curve.  Let
  \[
   j_{\mathrm H}\colon \mathbb P^1_{\upsilon}\longrightarrow\mathbb P^1_j
  \]
  be the exact degree-$63$ Hurwitz Belyi map.  Choose coprime polynomials
  $p_3,p_c\in\Q[\upsilon]$ such that
  \[
   j_{\mathrm H}(\upsilon)
     =\frac{256}{729}\frac{p_3(\upsilon)}{p_c(\upsilon)}.
  \]
  Thus the zeros of $p_3$ lie above $j=0$, while the zeros of $p_c$ are the
  finite poles of $j_{\mathrm H}$.  Define $p_2\in\Q[\upsilon]$ by
  \[
   j_{\mathrm H}(\upsilon)-1728
     =\frac{256}{729}\frac{p_2(\upsilon)}{p_c(\upsilon)}.
  \]
  Its zeros therefore lie above $j=1728$.  These two expressions are
  equivalent to the exact Belyi identity
  \begin{equation}\label{ll19:belyi-identity}
         4p_3-4p_2=19683p_c.
  \end{equation}
  With this normalization, the relevant numerator factorizations are
  \[
   p_3=A_{20}^3B_3,
   \qquad
   p_2=C_{31}^2(\upsilon-4),
  \]
  where the subscripts record the degrees of the factors. Recall that the inner Hurwitz line has equation $ \omega^2=\upsilon(\upsilon-4)$.
  Define
  \[
   a(\upsilon)
     =-3\upsilon A_{20}(\upsilon)B_3(\upsilon),
   \qquad
   b(\upsilon,\omega)
     =2\upsilon B_3(\upsilon)C_{31}(\upsilon)\omega.
  \]
  A direct calculation using \ref{ll19:belyi-identity} gives
  \[
   \frac{4a(\upsilon)^3}
   {4a(\upsilon)^3+27b(\upsilon,\omega)^2}
     =\frac{j_{\mathrm H}(\upsilon)}{1728}.
  \]
For a separable cubic $Y^3+aY+b$, the double cover
$v^2=Y^3+aY+b$ is branched at the three roots and $\infty$ and has
$j$-invariant $1728\cdot4a^3/(4a^3+27b^2)$. Conversely, this $j$-invariant
determines the projective equivalence class of the unordered four-point
branch configuration. Therefore the equality above shows that the roots of
  \[
   Y^3+a(\upsilon)Y+b(\upsilon,\omega)
  \]
  together with $\infty$ represent the branch configuration parametrized by
  the corresponding point of the Hurwitz curve.

  Using the parametrization \ref{eq:l19-inner-line} we put
  \begin{equation}
         \mathcal A(t)=t^{24}a(\upsilon(t)),
   \qquad
   \mathcal B(t)=t^{36}b(\upsilon(t),\omega(t)).
  \end{equation}
  The powers of $t$ clear the denominators introduced by the parametrization. They are compatible with the change of target coordinate
  $Y\mapsto t^{12}Y$, and therefore do not change the associated four-point
  configuration. We observe by substitution that $\mathcal A,\mathcal B$ are in $\Q[t]$. Thus the three finite branch values at the inner point $t$ are the roots of
  \[
   Y^3+\mathcal A(t)Y+\mathcal B(t).
  \]
  For the chosen target point $t_*=-45/13$, write
  \[
   (A_*,B_*)=(\mathcal A(t_*),\mathcal B(t_*)).
  \]
  We now describe the deformation chart.  For a general pair $(A,B)$, let
  $y_1,y_2,y_3$ be the roots of $Y^3+AY+B$. We seek a degree $20$ map to the $Y$-line whose fibers above $y_1,y_2,y_3$ have cycle type $2^{10}$ and whose fiber above infinity has cycle type $1^2 3^6$.  We represent such a map by
  \begin{equation}\label{eq:l19-target-equation}
   F(Y,z)=N_0(z)+YN_1(z)+Y^2N_2(z)=0,
  \end{equation}
  where projection to the coordinate $Y$ is the intended covering map.

  Introduce the cubic algebra
  \[
   \mathcal C_{A,B}
     =\Q[\theta]/(\theta^3+A\theta+B).
  \]
  The element $\theta$ simultaneously represents the three finite branch
  values.  Choose an element $\kappa\in\mathcal C_{A,B}$ and a monic
  degree $10$ polynomial $U\in\mathcal C_{A,B}[z]$. We now define
  $N_0,N_1,N_2$ by the identity 
  \begin{equation}\label{eq:l19-square-fibers}
   \kappa\,U(\theta,z)^2
     =N_0(z)+\theta N_1(z)+\theta^2N_2(z)
   \qquad\text{in }\mathcal C_{A,B}[z].
  \end{equation}
  Indeed, under the embedding that sends $\theta$ to $y_i$, this identity
  becomes
  \[
   F(y_i,z)=\kappa_i U_i(z)^2.
  \]
  Consequently, when $U_i$ is squarefree, the fiber above $y_i$ consists of
  ten points of ramification index $2$.  This gives cycle type $2^{10}$ at
  each of the three finite branch values.

  The fiber above $Y=\infty$ is controlled by the leading coefficient
  $N_2(z)$.  We impose
  \begin{equation}\label{eq:l19-order3}
   N_2=c_1Q_6^3Q_2,
  \end{equation}
  where $Q_6$ and $Q_2$ are monic of degrees $6$ and $2$.  Thus this fiber
  contains six points of ramification index $3$ and two unramified points,
  which is precisely the cycle type $1^2 3^6$.

  Finally, since \eqref{eq:l19-target-equation} is quadratic in $Y$, its
  discriminant with respect to $Y$ is
  \[
   D(z)=N_1(z)^2-4N_0(z)N_2(z).
  \]
  We impose (in order to enforce the normalization having genus $2$)
  \begin{equation}\label{eq:l19-genus2}
   D=c_2H_6R_{17}^2,
  \end{equation}
  where $H_6$ and $R_{17}$ are monic of degrees $6$ and $17$, respectively,
  and $H_6$ is required to be squarefree.  Completing the square in
  \eqref{eq:l19-target-equation} gives
  \[
   \bigl(2N_2Y+N_1\bigr)^2=D.
  \]
  After dividing by $R_{17}^2$, the function field of the normalization is generated by
  \[
   v^2=c_2H_6(z).
  \]
  Since $H_6$ is a squarefree sextic, this is a smooth curve of genus $2$.
  The factor $R_{17}^2$ has even multiplicity and introduces no additional
  branch points on the normalization.

  There are $66$ coordinates in this chart.  The element $\kappa$ and the
  monic polynomial $U$ contribute $3+30$ coordinates.  The polynomials
  $Q_6,Q_2$ and the scalar $c_1$ contribute $6+2+1$.  Finally,
  $c_2,H_6,R_{17}$ contribute $1+6+17$.  The polynomials $N_0,N_1,N_2$ are
  determined by \eqref{eq:l19-square-fibers} and are not additional
  coordinates.  Comparing coefficients in \eqref{eq:l19-order3} gives
  $21$ equations, while comparing coefficients in
  \eqref{eq:l19-genus2} gives $41$.  Hence the chart is defined by
  $62$ equations in $66$ coordinates.

  We work on the open subchart on which $Y^3+AY+B$ is separable, $N_1$ and $N_2$ are coprime,
  $\kappa$ is a unit, $c_1c_2\ne0$, each $U_i$ is squarefree, $H_6$ is
  squarefree, and $Q_6$ and $Q_2$ are squarefree and coprime.  On this subchart the three finite fibers have type $2^{10}$ and the fiber above
  infinity has type $1^2 3^6$. The normalization is the smooth
  genus-two curve $v^2=c_2H_6(z)$.  Exact calculation verifies that both
  the harmonic point and the target point lie in this open subchart.

We next place the harmonic cover in this chart. The harmonic cover has branch values
  \[
  s=0,\qquad s=\pm\sqrt{\frac{108}{19}},\qquad s=\infty.
  \]
  Under the change of coordinate $Y=2222316/s$, these become
  \[
  Y=\infty,\qquad
  Y=\pm2222316\sqrt{\frac{19}{108}},\qquad Y=0.
  \]
  Put $A_0=-868843330308$ and $B_0=0$. Since we have $-A_0=2222316^2\frac{19}{108}$, the three finite values are precisely the roots of $Y^3+A_0Y+B_0$.

   After making the substitutions $Y=2222316/s$ and $z=W$, the harmonic plane model supplies explicit values of all $66$ chart coordinates. Exact substitution verifies the equations
  \eqref{eq:l19-square-fibers}, \eqref{eq:l19-order3}, and
  \eqref{eq:l19-genus2}. Thus the harmonic cover gives an exact rational
  point of the deformation chart.

  Modulo $29$, the Jacobian of the $62$ defining equations at this point has
  rank $62$.  The chart is therefore smooth of dimension $4$ there.  We choose four coordinate functions whose differentials complement the
  Jacobian of the $62$ defining equations and fix their values. Adding these four conditions produces a square system of $66$ equations whose Jacobian is invertible
  modulo $29$.

  Since $t_*=-\frac{45}{13}\equiv1\pmod{29}$, the target branch pair $(A_*,B_*)$ reduces modulo $29$ to the harmonic
  branch pair $(A_0,B_0)$.  Multivariate Hensel lifting therefore gives a
  unique solution over $\Q_{29}$ with branch pair $(A_*,B_*)$, within the
  formal slice defined by the four chosen coordinates and in the modulo-$29$
  residue class of the harmonic point.

  It remains to prove that this $29$-adic point we just constructed descends to $\Q$.  We first
  reconstruct its absolute genus two invariants as rational numbers and verify the reconstructed relations exactly. Mestre's construction \cite{Mestre1991} associates to
  these invariants a conic over $\Q$.  In the present case this conic has an
  explicit rational point, which removes the obstruction to a genus-two
  model over $\Q$. Mestre's construction therefore produces a
  sextic model of the source curve over $\Q$.

 The rational sextic and the sextic obtained from the point produced by Hensel
  lifting represent the same genus-two curve, but use different coordinates on
  the hyperelliptic line. Modulo $29$, we enumerate all projective linear changes of
  coordinates that could identify their reductions. The change that
  identifies the two reductions is then lifted by a nonsingular four-variable
  Hensel Lemma calculation. The three variables describe the projective linear transformation,
  and the fourth records the proportionality factor between the two sextics. These computations can be found in the accompanying files.

  We then verify that the lifted transformation carries the marked divisors
  of degrees $2$ and $6$ above infinity to the corresponding divisors on the rational model. Rational reconstruction gives the $63$ coefficients of $N_0,N_1,N_2$.
  Substitution over $\Q$ confirms the branch pair $(A_*,B_*)$, the four fixed
  coordinates, and the chosen reduction modulo $29$. Hensel uniqueness under
  these conditions identifies this rational point with the lift obtained from
  the harmonic cover.

Let $(N_{0,*},N_{1,*},N_{2,*})$ denote the resulting rational coefficient rows. These rows and the corresponding pair $(A_*,B_*)$ are included in the
  accompanying files. Put
\[
 F_*(Y,z)=N_{0,*}(z)+YN_{1,*}(z)+Y^2N_{2,*}(z).
\]
This fixes a single cover. The stored rational point has branch pair
$(A_*,B_*)$, satisfies all $62$ identities and the same four fixed coordinates exactly,
and reduces modulo $29$ to the selected point, including the marked divisors $Q_2$ and $Q_6$.  It is therefore the Hensel lift in the chosen residue class and formal slice. 

For $F_*$, the branch cubic is irreducible and squarefree, the
three involution fibers are exact squares of degree $10$, $N_2$ factors as
$Q_6^3Q_2$, and $H_6$ is squarefree.  The visible ramification is
\[
 3\cdot10+6(3-1)=42=2\cdot20+2\cdot2-2 ,
\]
and the source cuve has genus $2$ so Riemann--Hurwitz shows that these four fibers are the whole branch locus.

\begin{proposition}\label{prop:l19-target-cover}
The cover $F_*(Y,z)=0$ has smooth genus-two source and
geometric monodromy $\PSL_2(\F_{19})$.
\end{proposition}

\begin{proof}
The identities above along with the Riemann--Hurwitz formula prove smoothness, the genus, and completeness of the branch locus.  For the monodromy,
note that the harmonic point and the target reduce to the same nonsingular
point of the chart modulo $29$. Away from the four disjoint branch
sections, the degree $20$ family is finite \'etale.  Any monodromy group here
is a subgroup of $S_{20}$, so its order divides $20!$ and is prime to $29$.
Therefore specialization preserves the monodromy
\cite{SGA1}*{Expos\'e~XIII, Cor.~2.9 and \S2.10} and both
characteristic zero fibers have the monodromy of the common special fiber.
By Lemmas~\ref{lem:l19-sign-pullback}
and~\ref{lem:l19-harmonic-is-yang}, and the identification of $W$ made at the
end of \S\ref{subsec:l19-sign}, that group is $\PSL_2(\F_{19})$.

\end{proof}

\begin{remark}
    In the language of Appendix~\ref{app:f13-construction}, the covers in this
chart lie in the Nielsen class of $\PSL_2(\F_{19})$ with three branch cycles
of order $2$, of cycle type $2^{10}$, and one of order $3$, of cycle type
$1^2 3^6$.  The analog of this class at $\ell=13$ is $(2A,2A,2A,3A)$, and
that is precisely the class K\"onig rules out~\cite{Koenig2017}*{\S7}. In that case, 
the source curves have genus $0$, however the Hurwitz curve is an elliptic curve of
rank $0$ whose rational points give no covers with real fibers.  At
$\ell=13$, we therefore moved to $(2A,2A,2A,6A)$ and genus one sources. At $\ell=19$, we use the class $(2A,2A,2A,3A)$, in which the Hurwitz curve has genus $0$ and the source curves have genus $2$. 
\end{remark}

\subsection{The Galois group of \texorpdfstring{$f_{19}$}{f19}}
\label{subsec:l19-specialization}
We specialize $F_*$ at
\[
 Y_0=-7546608500347.
\]
We then apply a change of the primitive element for the specialization and follow it by polynomial reduction. The result is the monic polynomial $f_{19}$ of
Appendix~\ref{app:f19}. Using a minimal polynomial identity, we have verified that $f_{19}$ and $F_{*}(Y_0,z)$ define isomorphic degree $20$ fields.

We chose the point $Y_0$ as follows. The three roots of $Y^3+A_*Y+B_*$ divide the real $Y$-line into four intervals. An exact Sturm calculation~\cite{BasuPollackRoy2006}*{\S2.2.2} shows that only one of these intervals gives specializations with twenty real roots. We searched this interval subject to congruence conditions at $3$ and $17$ chosen to produce favorable local data for B\"oge's criterion. The Chinese remainder theorem reduces the search to an arithmetic progression which contains $Y_0$.
\begin{proposition}\label{prop:l19-projective}
The polynomial $f_{19}$ is irreducible, all twenty of its roots are real,
and
\[
 \Gal(L_{19}/\Q)\cong\PSL_2(\F_{19})
\]
in the natural action on $\PP^1(\F_{19})$.
\end{proposition}

\begin{proof}
Exact factorization over $\F_7$ and over $\F_{23}$ gives
\begin{align*}
 f_{19}\bmod7={}&
 (x^5+x^3+2x^2+4x+2)(x^5+5x^4+5x^3+5x+4)\\
 &\cdot(x^5+6x^4+2x^3+5x^2+4x+6)(x^5+6x^4+4x^3+3x^2+5x+6),\\[1mm]
 f_{19}\bmod23={}&(x+17)
 (x^{19}+2x^{18}+12x^{17}+4x^{16}+16x^{15}+6x^{13}+x^{11}\\
 &\hspace{16mm}{}+20x^9+x^8+14x^7+7x^6+16x^5+9x^4
 +21x^3+7x^2+5x+22),
\end{align*}
both squarefree with every displayed factor irreducible.  The degrees of the
factors of a rational factor of $f_{19}$ would have to add up to the same
number in both patterns. The subset sums of $\{5,5,5,5\}$ are
$0,5,10,15,20$ and those of $\{1,19\}$ are $0,1,19,20$, whose only common
values are $0$ and $20$.  Hence $f_{19}$ is irreducible.  By Dedekind's theorem its Galois group $G$ contains an element of cycle type
$1\cdot19$, hence an element of order $19$.  

Irreducibility shows that $G$ is transitive.  Let
$H=\PSL_2(\F_{19})$ in its natural degree-20 action, and let $A$ be the
arithmetic monodromy group of $F_*$.  Proposition~\ref{prop:l19-target-cover}
gives the geometric monodromy group $H$.  Hence
\[
 H\trianglelefteq A\leq N_{S_{20}}(H).
\]
 The exact normalizer calculation gives
\[
 N_{S_{20}}(H)=\PGL_2(\F_{19}),\qquad
 N_{S_{20}}(H)\cap A_{20}=H.
\]
Since $Y_0$ is not a branch point of the family, specialization at $Y_0$
embeds $G$ into the arithmetic monodromy group \cite{Serre2008}*{Ch.~1}. We also check that the
discriminant of $f_{19}$ is a nonzero square, so $G\leq A_{20}$ and therefore $G\leq H$.

For the lower bound, an exhaustive search inside $H$ shows that no proper
subgroup of $H$ containing an element of order $19$ is transitive.  We conclude $G=H$. 

Additionally, the exact Magma script in the level $19$ code package proves irreducibility and, using \texttt{GaloisProof},
identifies the Galois group in its degree-$20$ action with
$\PSL_2(\F_{19})$. Sturm arithmetic in the accompanying
Sage file gives twenty real roots~\cite{BasuPollackRoy2006}*{\S2.2.2}.
\end{proof}

\subsection{Ramification}\label{subsec:l19-local}
Write $K_{19}=\Q[x]/(f_{19})$, put $p_3=5605357159$, and let
\begin{equation*}
\begin{split}
 p_4={}&1364622470507035764968656932636748268884643728575565929266173990\\
 &856504046452283837355063452317992591 ,
\end{split}
\end{equation*}
a prime of $100$ digits.  A certified maximal-order computation gives
\begin{equation}\label{eq:l19-field-discriminant}
 \Disc(K_{19})=17^{10}\,p_3^{16}\,p_4^{10} .
\end{equation}

\begin{lemma}\label{lem:l19-local}
The primes ramified in $L_{19}$ are exactly $17$, $p_3$ and $p_4$, and for
each the pair $(e(q),f(q))$ is as in Table~\textup{\ref{tab:l19-local}}.
\end{lemma}
  \begin{table}[ht]
  \centering
  \begin{tabular}{@{}cclc@{}}
  \toprule
  $q$ & type of $q$ in $K_{19}$ & $(I,D)$ & $(e(q),f(q))$\\
  \midrule
  $17$ & $(2,5)^2$ & $(C_2,C_{10})$ & $(2,5)$\\
  $p_3$ & $(1,1)^2,(9,1)^2$ & $(C_9,C_9)$ & $(9,1)$\\
  $p_4$ & $(2,2)^5$ & $(C_2,C_2\times C_2)$ & $(2,2)$\\
  \bottomrule
  \end{tabular}
  \caption{Local data for $K_{19}$ and for its Galois closure $L_{19}$.}
  \label{tab:l19-local}
  \end{table}

\begin{proof}

As in Lemma \ref{lem:l13-local}, we construct the maximal order of $K_{19}$ using PARI, and certify it with \texttt{nfcertify}, and also confirm the discriminant \eqref{eq:l19-field-discriminant}. Exact factorization of the primes $17\mathcal O_{K_{19}}$, $p_3\mathcal O_{K_{19}}$ and $p_4\mathcal O_{K_{19}}$ gives the decomposition types recorded in the second column of Table~\ref{tab:l19-local}.
All three ramified primes are tame,
so the enumeration of pairs $I\trianglelefteq D\leq\PSL_2(\F_{19})$
described after Lemma~\ref{lem:orbit-dictionary} applies with $I$ cyclic,
$D/I$ cyclic. In each case, we find a single conjugacy class in the stated permutation action. Every prime ramified in $K_{19}$ ramifies in $L_{19}$, and by
Lemma~\ref{lem:no-hidden-ramification} there is no further ramification, proving the assertion.
\end{proof}

The prime $p_3$ has $e(p_3)=9$ odd and so imposes no condition in
Theorem~\ref{thm:boege}.  The other two give
\[
 e(17)=2,\quad f(17)=5\ \text{odd},\quad 17\equiv1\pmod4,
\]
and
\[
 e(p_4)=2,\quad f(p_4)=2\ \text{even},\quad p_4\equiv3\pmod4 .
\]
Both satisfy the parity condition, so Lemma~\ref{lem:l19-local} and
Proposition~\ref{prop:l19-projective} together prove
Theorem~\ref{thm:stems} for $\ell=19$.  This completes the proof of
Theorem~\ref{thm:main}.
\qed

\section{A tower of totally real congruence
  extensions}\label{sec:consequence}

The representations produced by Theorem~\ref{thm:main} are even representations which are known to be more difficult to study. In two papers, \cites{Ramakrishna1999,Ramakrishna2002}, Ramakrishna investigated obstructions to lifting mod $p$ Galois representations to characteristic $0$ $p$-adic representations. His deformation theorem, combined with Theorem \ref{thm:main}, gives the following
consequence.

\begin{corollary}\label{cor:adic-towers}
For each $\ell\in\{13,19\}$ there is a continuous surjection
\[
 \rho_\ell\colon\GQ\twoheadrightarrow\SL_2(\Z_\ell)
\]
unramified outside a finite set of primes, whose reduction modulo $\ell$ cuts
out $M_\ell$, for which every complex conjugation lies in $\ker\rho_\ell$.
Consequently, the fixed field $M_{\ell,n}$ of
$\ker(\rho_\ell\bmod\ell^n)$ is totally real with
$\Gal(M_{\ell,n}/\Q)\cong\SL_2(\Z/\ell^n\Z)$, the fields satisfy
$M_{\ell,n}\subset M_{\ell,n+1}$, and all are unramified outside one fixed
finite set.
\end{corollary}

\begin{proof}
Fix an isomorphism $\Gal(M_\ell/\Q)\simeq\SL_2(\F_\ell)$ and let
$\overline\rho_\ell\colon\GQ\twoheadrightarrow\SL_2(\F_\ell)$ be the
corresponding representation.  Ramakrishna proved that for a finite field $k$
of characteristic at least $7$, every surjection
$\GQ\twoheadrightarrow\SL_2(k)$ admits a deformation to
$\SL_2(W(k))$ that is unramified outside a finite set
\cite{Ramakrishna2002}*{Corollary~1(a)}. No parity hypothesis on the residual
representation is imposed.  Since $W(\F_\ell)=\Z_\ell$, this gives
$\rho_\ell$.

For $\ell\geq5$, a closed subgroup of $\SL_2(\Z_\ell)$ that surjects
modulo $\ell$ is the whole group
\cite{Serre1968}*{Ch.~IV, Lemma~3}, so the lift is surjective. Let $c$ be a complex conjugation. The total reality of $M_\ell$ gives
$\overline\rho_\ell(c)=I$, so $\rho_\ell(c)\equiv I$ modulo $\ell$ and
$\rho_\ell(c)+I$ is invertible. Since $c^2=1$, we also have
$\rho_\ell(c)^2=I$. Therefore
\[
 \bigl(\rho_\ell(c)-I\bigr)\bigl(\rho_\ell(c)+I\bigr)=0 ,
\]
which forces $\rho_\ell(c)=I$.  Reducing modulo $\ell^n$ gives the stated
finite quotients, all totally real and all unramified outside the
ramification set of $\rho_\ell$.  The kernels decrease with $n$, so their
fixed fields are nested.
\end{proof}

\begin{remark}
Corollary~\ref{cor:adic-towers} is a pure existence result.  It produces no defining
polynomials for the higher layers, and we do not assert that Ramakrishna's
even lifts are potentially semistable or geometric.
\end{remark}

\section{The computations}\label{sec:computations}

All the necessary code can be found in the following GitHub repository.
\begin{center}
    \href{https://github.com/eekarabiyik/SL2F13}{https://github.com/eekarabiyik/SL2F13}.
\end{center}

This work includes a large amount of computer assisted calculations. These consist of factorizations
of integers and of polynomials over $\Q$ and over finite fields, Sturm
counts~\cite{BasuPollackRoy2006}*{\S2.2.2}, maximal orders and their prime factorization data, resolvent
constructions, exhaustive searches in finite permutation groups, and
substitutions into polynomial identities. 
We used SageMath~10.9~\cite{SageMath} with PARI/GP~2.17.4~\cite{PARI2} for the
arithmetic, and GAP~4.16.0~\cite{GAP} with the transitive groups
library~\cite{TransGrp} for the group theory.  Every maximal order is
certified by PARI's \texttt{nfcertify}. None of the discriminants or
prime factorizations used in this paper is conditional on the generalized Riemann
hypothesis.  Two Magma~\cite{Magma} scripts verify Propositions \ref{prop:l13-projective} and \ref{prop:l19-projective}.

To find the equations of
Appendix~\ref{app:f13-construction} and \S\ref{subsec:l19-deformation}, our computations 
involved numerical continuation, $p$-adic lifting, Berlekamp--Massey and
Pad\'e reconstruction, and rational reconstruction. The resulting candidates
were retained only after exact substitution into their defining identities.

The archive in the GitHub repository contains one self-contained package for each prime.  Each has a short certificate that begins with the displayed coefficients and verifies everything asserted about $L_\ell$ in \S\ref{sec:l13} and
\S\ref{sec:l19}.  The longer level-$13$ check verifies the exact degree-$54$ Belyi identity, the stored genus-one family, the specialization at $t_0$, and the relation with the displayed polynomial.  It does not repeat the numerical continuation and Hensel computations by which the stored family was discovered.  The longer level-$19$ check reconstructs Yang's equation from Hecke operators, the harmonic model, the sign pullback, and the $29$-adic descent.

\appendix

\section{The construction of \texorpdfstring{$f_{13}$}{f13}}
\label{app:f13-construction}
This appendix records how $f_{13}$ was computed. 

\subsection{The Hurwitz component}
Let $G=\PSL_2(\F_{13})$ act on fourteen points.  A four-point $G$-cover of
$\PP^1$ is described topologically by branch cycles $g_1,g_2,g_3,g_4\in G$
with
\[
 g_1g_2g_3g_4=1,\qquad\langle g_1,g_2,g_3,g_4\rangle=G ,
\]
taken up to simultaneous conjugation. The resulting finite set is the inner
Nielsen class. The Hurwitz braid group acts on it by moving the branch
points. The orbits of this action are the connected components of the corresponding
Hurwitz space.  We refer to \cite{Koenig2017}*{\S2} for further details including the algorithms used.

We take the Nielsen class
\begin{equation}\label{eq:l13-nielsen}
 g_1,g_2,g_3\in2A,\qquad g_4\in6A ,
\end{equation}
whose two conjugacy classes act on $\PP^1(\F_{13})$ with cycle types $1^2 2^6$ and
$1^2 6^2$.  The permutation indices sum to
$3\cdot6+10=28=2\cdot14+2\cdot1-2$, so by Riemann--Hurwitz the degree $14$
source curves have genus~$1$. Note that K\"onig investigates a different obstructed genus-zero
family of type $(2A,2A,2A,3A)$~\cite{Koenig2017}*{\S7}. 

For one fixed ordering, fixing $g_1$ leaves $1296$ tuples in
\eqref{eq:l13-nielsen}. The centralizer of $g_1$, of order $12$, acts freely.
Indeed, an element fixing a generating tuple centralizes all its entries and
therefore lies in the trivial center of $G$. Thus there are $108$ inner
classes.  Let $q_1,q_2,q_3$ be the standard Hurwitz generators,
where $q_i$ interchanges the $i$th and $(i+1)$st branch points, and put
\[
 Q'=\bigl\langle q_1q_3^{-1},(q_1q_2q_3)^2\bigr\rangle .
\]
This is the standard Klein-four reduction subgroup, its image consists of the identity and the three double
transpositions of the four branch labels of our Nielsen class.  Its orbits identify the four
markings of a branch configuration having the same cross ratio. To apply
the full braid group, we use all four placements of the $6A$ class,
obtaining $432$ inner states. The group $Q'$ has free orbits of size $4$ so we are again left with $108$ reduced inner states.  The outer involution of $\PGL_2(\F_{13})/G$ identifies these in
pairs, leaving $54$ reduced absolute states. These states are the sheets of
the map from the absolute Hurwitz curve to the branch-configuration line.
Thus this map is a degree-$54$ Belyi map. Braid enumeration gives its
three-point monodromy triple with cycle types
\begin{equation}\label{eq:l13-hurwitz-passport}
 3^{18},\qquad 1^2 2^{26},\qquad 2\,3^4\,6\,7^3\,13 .
\end{equation}
The indices sum to $36+26+44=106=2\cdot54 - 2$, so the compactified absolute
Hurwitz curve has genus $0$. Its cusps of widths $2$, $6$ and $13$ are
rational and unique. We normalize a coordinate $x$ by placing them at
$0$, $8$ and $\infty$, respectively.

The reconstruction produced the following exact degree-$54$ Belyi map with
passport \eqref{eq:l13-hurwitz-passport}:
\begin{equation}\label{eq:l13-belyi}
 \beta=-\frac{\Phi^3}{2^6 3^9\,c},
 \qquad
 \beta-1=-\frac{\Psi^2\lambda}{2^6 3^9\,c},
\end{equation}
where $\Phi,\Psi\in\Z[x]$ are monic of degrees $18$ and $26$, given in GitHub, and
\[
\begin{aligned}
 \lambda&=x^2-10x+13, &
 \sigma&=x^3-11x^2+24x-13,\\
 \varrho&=x^4-19x^3+113x^2-221x+169, &
 c&=x^2(x-8)^6\varrho^3\sigma^7 .
\end{aligned}
\]
The two expressions in \eqref{eq:l13-belyi} are compatible from the equation
\begin{equation}\label{eq:l13-belyi-identity}
 \Phi^3-\Psi^2\lambda=-2^6 3^9 c
\end{equation}
holding in $\Z[x]$. Further computations show that the zero and pole divisors of $\beta$ and $\beta -1$ give the three cycle types of \eqref{eq:l13-hurwitz-passport}. Riemann--Hurwitz then
shows that $\beta$ has no other critical points.  The map $\beta$ was found as follows. We first
normalized the three distinguished cusps. The resulting equations were then solved
numerically to high precision and then the rational coefficients were recognized.  A separate path continuation matched the monodromy triple, up to
simultaneous conjugacy, with the enumerated braid triple whose cycle types are given in \eqref{eq:l13-hurwitz-passport}.

\subsection{Constructing the degree \texorpdfstring{$14$}{14}
family}
  Under the numerical identification just described, the two simple points of $\beta$ above $1$, namely the roots of
  $\lambda(x)$, form the branch divisor of the inner-to-absolute double cover.
  Its function field therefore has the form
  \[
   \Q(x)\bigl(\sqrt{d\,\lambda(x)}\bigr),
   \qquad d\in\Q^\times/\Q^{\times2}.
  \]
  The rational width-two cusp lies above $x=0$. Since $\lambda(0)=13$, the
  existence of a rational point above $x=0$ implies that $13d$ is a square in
  $\Q$. Thus \(d\) has square class \(13\). After rescaling the second
  coordinate, the inner curve has the equation
  \[
   y^2=13\lambda(x)=13(x^2-10x+13).
  \]
  The rational point $(0,13)$ gives the parametrization
\[
 x(\eta)=-\frac{26\eta}{\eta^2-10\eta+12},
 \qquad
 y(\eta)=\frac{13(12-\eta^2)}{\eta^2-10\eta+12}.    
\]
The deck involution in this parametrization becomes $\eta\mapsto12/\eta$.  We have used numerical constructions to obtain the above identities, but we also verify all the equations directly over $\Q$.  Above the inner point
$\eta$, the three order 2 branch points of the associated degree $14$
cover are the roots of
\[
 b_\eta(T)=T^3-39\,\Phi(x(\eta))\,T+26\,\Psi(x(\eta))\,y(\eta) .
\]
To see that this cubic represents the required point of the moduli space,
apply the four-point classification used in \S\ref{subsec:l19-deformation}.
The double cover branched at the roots of $T^3+pT+q$ and at $\infty$ has
normalized $j$-invariant $4p^3/(4p^3+27q^2)$; substituting
$p=-39\Phi$ and $q=26\Psi y$ and using $y^2=13\lambda$ and
\eqref{eq:l13-belyi-identity} returns $\beta(x)$. 

We work at $\eta=1$, where $x=-26/3$ and $y=143/3$, so that
\begin{equation*}
\begin{split}
 b_1(T)={}&T^3
 -\frac{304640314484285739614258817616}{129140163}\,T\\
 &+\frac{336287756831510666288167724601841306395891328}
        {7625597484987} .
\end{split}
\end{equation*}
This cubic is irreducible with positive discriminant, its roots are real
and distinct.    We seek a degree-$14$ cover whose three finite branch values are the
  roots $t_1,t_2,t_3$ of $b_1(T)$.  Above each $t_i$ we require six points of
  ramification index $2$ and two unramified points. The two unramified points
  in the prescribed fiber above $T=\infty$ form an effective divisor of
  degree $2$ on the genus-one source. Riemann--Roch gives a nonconstant
  function $z$ with this polar divisor. Hence the function field of the
  source is quadratic over $\Q(z)$, so $T$ satisfies a quadratic equation
  over $\Q(z)$. We use this choice of $z$ throughout the deformation. To
  impose the three finite factorizations simultaneously, put
  \[
   \mathcal C_1=\Q[\theta]/(b_1(\theta))
  \]
  and look for a cover of the form
  \begin{equation}\label{eq:l13-family}
 F(T,z)=N_0(z)+TN_1(z)+T^2N_2(z),
 \qquad
 (\deg N_0,\deg N_1,\deg N_2)=(14,13,12),
\end{equation}
  and require
  \[
   F(\theta,z)=\kappa(\theta)U(\theta,z)^2V(\theta,z)
   \qquad\text{in }\mathcal C_1[z],
  \]
  where $U$ and $V$ are monic of degrees $6$ and $2$.  Under an embedding
  $\theta\mapsto t_i$, this becomes
  \[
   F(t_i,z)=\kappa_iU_i(z)^2V_i(z).
  \]
  When $U_i$ and $V_i$ are squarefree and coprime, the six roots of $U_i$
  give the six ramification points of index $2$, while the two roots of
  $V_i$ give the two unramified points.  Thus the fiber has cycle type
  $1^2 2^6$.

  At $T=\infty$ we require two ramification points of index $6$ and two
  unramified points.  Writing $S=1/T$ and multiplying the equation by $S^2$
  shows that the finite part of the fiber at infinity is cut out by
  $N_2(z)$.  We therefore impose
\begin{equation}\label{eq:l13-infinity-factorization}
       N_2(z)=Q_\infty(z)^6,\qquad \deg Q_\infty=2.
\end{equation}
With $S=1/T$, local homogenization shows that the two distinct roots of the
squarefree polynomial $Q_\infty$ yield two branches of ramification index
$6$, while the degree drops of $N_1$ and $N_2$ yield two unramified branches
over $z=\infty$. Thus the fiber has cycle type $1^2 6^2$.

Since
\eqref{eq:l13-family} is quadratic in $T$, projection to the $z$-line makes
the normalization of $F=0$ hyperelliptic, and the genus one condition (branching in exactly four points) reads
\begin{equation}\label{eq:l13-source-discriminant}
 N_1(z)^2-4N_0(z)N_2(z)=d_0\,H_4(z)\,R_{11}(z)^2
\end{equation}
with $H_4$ monic squarefree of degree $4$ and $R_{11}$ monic of degree $11$. The desired domain is then the curve given by $Y^2=d_0H_4(z)$.

We work on the open subchart on which $\kappa$ is a unit, $d_0\ne0$, each
pair $U_i,V_i$ is squarefree and coprime, and $Q_\infty$ and $H_4$ are
squarefree.  Since $b_1$ is separable, these conditions give the stated
fiber types and a smooth genus-one normalization.

Expanding everything in the basis $1,\theta,\theta^2$ leaves
\[
 \underbrace{3}_{\kappa}+\underbrace{6\cdot3}_{U}
 +\underbrace{2\cdot3}_{V}+\underbrace{2}_{Q_\infty}
 +\underbrace{1}_{d_0}+\underbrace{4}_{H_4}
 +\underbrace{11}_{R_{11}}=45
\]
unknowns.  Coefficient comparison in
  \eqref{eq:l13-infinity-factorization} and
  \eqref{eq:l13-source-discriminant} gives $44$ scalar equations.  At the
  harmonic seed their Jacobian has rank $43$, because one of the $44$
  differential rows lies in the span of the others.  The resulting
  two-dimensional freedom among the $45$ unknowns is exactly the freedom to
  translate and rescale the auxiliary coordinate $z$.  For the local
  continuation, we retain $43$ independent equations and fix the constant and
  linear coefficients of $Q_\infty$ at their seed values.  The resulting
  $45$ equations have a nonsingular Jacobian at the seed, and every exact
  endpoint is checked against all $44$ original equations.

  After continuing to $\eta=7/2$, we make an affine change of $z$ and impose
  \[
   \operatorname{coef}_{Q_\infty}(z)=0,\qquad
   \operatorname{coef}_{H_4}(z^3)=-4,
  \]
  where $\operatorname{coef}_f(z^k)$ denotes the coefficient of $z^k$ in $f$.
  The first condition fixes the translation of $z$, and the second fixes its
  scaling.  This normalization is used in the subsequent reconstruction.
A starting point comes from the harmonic point $\eta=2\delta$ of the inner conic where
$\delta=\sqrt3$.  Let $w$ be the coordinate  of the codomain of the harmonic cover and $X$ the affine
coordinate in the plane model of the domain. The equation
\[
 457w^2X^{12}A_2(X)-(493-88\delta)B_2(X)C_3(X)^4=0
\]
with
\[
\begin{aligned}
 A_2(X)&=481X^2+(123\delta+180)X+(594-27\delta),\\
 B_2(X)&=457X^2-(213\delta+648)X+(594+81\delta),\\
 C_3(X)&=1403X^3-(2703\delta+153)X^2-(1458\delta+1512)X-(4293\delta+243)
\end{aligned}
\]
defines a cover with the required fibers. Its connectedness is checked by
exact factorization, its discriminant is checked to be a square, and its
normalized source has the genus-one model
$Y^2=\frac{493-88\delta}{457}A_2(X)B_2(X)$.  On this normalization the two
coordinates are related, after choosing the sign of $Y$, by
\[
 w=\frac{Y C_3(X)^2}{X^6A_2(X)}.
\]
The projection from the local solution curve to the $\eta$-line has two local branches. Numerical continuation from the harmonic solution chose one of the two branches and
  proposed a rational point above $\eta=7/2$ (chosen because of proximity to $2\sqrt{3}$). We verified by exact substitution that this rational point satisfies the complete system.
  Starting from this point, we computed the solution in a finite field power series ring modulo $1000003$ to precision $256$.  Berlekamp--Massey and Pad\'e reconstruction then
  produced a candidate above $\eta=1$, where the branch cubic is irreducible.  We used Hensel's lemma to lift this candidate at fixed $\eta=1$ and reconstructed all $45$
  coordinates over $\Q$.  We then verified the complete system is satisfied. Hence, the finite field continuation was only used to locate the candidate at $\eta=1$.

At the exact solution, $d_0$ is nonzero, $H_4$ is squarefree, the three
pairs $U_i,V_i$ are squarefree and coprime, every $\kappa_i$ is nonzero,
and $Q_\infty$ is squarefree.  Thus
the three finite fibers have type $1^2 2^6$, and the fiber above infinity
has type $1^2 6^2$. The visible ramification is
$3\cdot6+10=28$, the full Riemann--Hurwitz total for a degree-$14$ map from
a genus-one curve, so there are no other branch points. The polynomial $F(0,z)$ is
irreducible, and since the leading coefficient of $F$ in $z$ is a nonzero
constant, $F$ is irreducible over $\Q(T)$.

\subsection{Specializing}
Let $\tau_1<\tau_2<\tau_3$ be the roots of $b_1$.  Specializing $T$ inside
the interval $(\tau_2,\tau_3)$ was numerically observed to keep all fourteen
roots real. Numerically, we have 
\[
 \tau_2\approx2.8040063127\cdot10^{10},
 \qquad
 \tau_3\approx2.8043085841\cdot10^{10} .
\]
For an integer $t$ in that interval let $g_t(z)\in\Z[z]$ be the primitive
part of $N_0(z)+tN_1(z)+t^2N_2(z)$, that is, the integral polynomial obtained
by clearing denominators and dividing out the content.  We
chose
\[
 t_0=28040147150 ,
\]
for which
\begin{equation}\label{eq:l13-branch-factorization}
 b_1(t_0)=-\frac{2^3\cdot5119\cdot p_1\cdot p_2}{7625597484987} .
\end{equation}
The three odd factors are prime and congruent to $3$ modulo $4$, which is
what made this specialization a promising candidate. Their factorization
type in the stem predicted the favorable local row $(e,f)=(2,2)$.   

Our code checked that $g_{t_0}$ is primitive and irreducible with
fourteen real roots and square discriminant.  It is not monic, and its
coefficients are large. Applying PARI's exact \texttt{polredabs} to
$\Q[z]/(g_{t_0})$ produces the monic polynomial $f_{13}$ of
Example~\ref{ex:f13}. 

\section{The degree \texorpdfstring{$20$}{20} projective stem field}
\label{app:f19}

The coefficients of $f_{19}$ are listed below. 

\input{polynomial19}

\bibliographystyle{amsra}
\bibliography{references}

\end{document}

%% file: polynomial13.tex
\begin{equation}\label{eq:f13-polynomial}
\begin{aligned}
f_{13}(x)={}&x^{14}-6x^{13}-204840441x^{12}-1119125930364x^{11}\\
 &+4529702096707905x^{10}+38159749624959546246x^{9}\\
 &-18546069574715570089137x^{8}-452968533089504837361085896x^{7}\\
 &-192666699116684784351337401333x^{6}+2209557167547162714402116518030950x^{5}\\
 &+1429616844522426330943895770324225533x^{4}\\
 &-3899920152435181455812963700044487447036x^{3}\\
 &-1769134263978452128562169254872450324635597x^{2}\\
 &+1792548173476070932710559372955582667141406554x\\
 &+825828856675546377159985433494767986559465932589
\end{aligned}
\end{equation}

%% file: polynomial19.tex
\begin{longtable}{@{}r>{\raggedright\arraybackslash\scriptsize}p{0.86\textwidth}@{}}
\toprule
$i$ & $a_i$ in $f_{19}(x)=\sum_{i=0}^{20}a_i x^i$\\
\midrule
\endfirsthead
\toprule
$i$ & $a_i$ (continued)\\
\midrule
\endhead
20 & $1$\\
19 & $-4$\\
18 & $-2873302382095736645289420504\allowbreak{}3351419234$\\
17 & $-1427370795596217703865009004\allowbreak{}4958006232528983921598922884$\\
16 & $3234440741404879449518318757\allowbreak{}6247912490430844121875890141\allowbreak{}3397892565191498577$\\
15 & $3050574465674072484837814059\allowbreak{}2050350959697703697714600576\allowbreak{}0229835906887688529279766318\allowbreak{}813484256$\\
14 & $-1744002606747923220693690837\allowbreak{}6859538853684191246912268867\allowbreak{}4167197048109754138320199820\allowbreak{}8160698551186713991165601663$\\
13 & $-2384962492143838995704216764\allowbreak{}8259429335783437647978229458\allowbreak{}9147636616153786643143294681\allowbreak{}8411916957826188752378533237\allowbreak{}205702306427313102$\\
12 & $4353796849873165112332174503\allowbreak{}4355777967813104575455427020\allowbreak{}8107912494622526863957689178\allowbreak{}7605723687260658553199927701\allowbreak{}9693406904823365766542431582\allowbreak{}39662525$\\
11 & $8301040870938150498053122726\allowbreak{}5250461515421868808923242394\allowbreak{}6167546344286008715790002425\allowbreak{}1886021511695622475860085614\allowbreak{}2442510345780928783542394896\allowbreak{}12807969922580775883414124$\\
10 & $-3526270265109903895556537490\allowbreak{}1075419355338587374813496266\allowbreak{}7772517528481585433066679959\allowbreak{}7018304313931427338706924528\allowbreak{}3554338985522344870936265890\allowbreak{}5393556357640534211179774612\allowbreak{}4098936390694883$\\
9 & $-1262814336026952044367203309\allowbreak{}4048820308466567791132728996\allowbreak{}4387262355423525583695519294\allowbreak{}6070463172834872377657343895\allowbreak{}5412027970635042692335049461\allowbreak{}3007518879362264429384071339\allowbreak{}7889764332995692191021750836\allowbreak{}3409506$\\
8 & $-1891861718674656333207943760\allowbreak{}1272757879967489437673871227\allowbreak{}3599500774057953890441180506\allowbreak{}6529420408091294167772388443\allowbreak{}8461079898125028727036698279\allowbreak{}9776853553666667116894366466\allowbreak{}7868311329630179935395645786\allowbreak{}234250208982368377058208$\\
7 & $7572122261101128567080492133\allowbreak{}6581018846867438629273858555\allowbreak{}9907492228686700587504977704\allowbreak{}8358492688015631731897335033\allowbreak{}3233121150966393713644550501\allowbreak{}7611514755525392311696204479\allowbreak{}7268363238094259672755489687\allowbreak{}9465939454931012735877949352\allowbreak{}08669609223808$\\
6 & $2534096526078379799100670729\allowbreak{}8893603224325321460383056639\allowbreak{}6304215409130032545919782243\allowbreak{}9614400928457966429498291049\allowbreak{}6708041031126186112588766119\allowbreak{}9906668833150913450927852470\allowbreak{}3475993764481360353112568014\allowbreak{}2149448078380880619521060437\allowbreak{}1499247082778443689638428035\allowbreak{}7719$\\
5 & $-2072824901948810594831757322\allowbreak{}6020236205308944679422307730\allowbreak{}0223669846940903139407580329\allowbreak{}5758260097832113331329564034\allowbreak{}4717408058014281647360584855\allowbreak{}7174992520983155090948069548\allowbreak{}1210047844969825721239953088\allowbreak{}6796018525239184068390092906\allowbreak{}3477249023679627404312261077\allowbreak{}9765153182502638752730$\\
4 & $-7534025836705874453136560413\allowbreak{}2239536629815665972744502533\allowbreak{}0292166111420205891292497491\allowbreak{}9055780696526217802105688313\allowbreak{}3800617553189794895238213572\allowbreak{}9873437769366970528533538433\allowbreak{}6677624063894496354230470925\allowbreak{}1142445351129533290696614486\allowbreak{}9180591936747986477780435563\allowbreak{}6499806671054939288853832817\allowbreak{}51156681596$\\
3 & $2773086799438814928734141334\allowbreak{}7498663452262152066389604634\allowbreak{}9478111843531014543126871914\allowbreak{}3035739044778898445439480103\allowbreak{}4742024420824833943883769462\allowbreak{}2537839955544720182837992411\allowbreak{}9298037716460298818843628127\allowbreak{}0458065708097199528078397461\allowbreak{}9249771127196777012974122486\allowbreak{}8785465142205229388165452989\allowbreak{}4037801347229592371152109800\allowbreak{}0$\\
2 & $7319961867596505741522531044\allowbreak{}6734247404759384078419646840\allowbreak{}0232735287564552166551775696\allowbreak{}4649497934114867183793902775\allowbreak{}8946889828763829532504942515\allowbreak{}6681694787425216532441363212\allowbreak{}7539975538444309321899811286\allowbreak{}5971851154241584889382914118\allowbreak{}8201023392179182957250899738\allowbreak{}0996977640851080520388034613\allowbreak{}9090908763224267397395143104\allowbreak{}132132020210800400$\\
1 & $-1510220604850384718215007913\allowbreak{}2434191909001722009385462377\allowbreak{}8721541550814125135001812846\allowbreak{}3330697954533691612334862540\allowbreak{}4321608821344013209058031678\allowbreak{}7086474103391121282087725878\allowbreak{}3616375093579470272232271118\allowbreak{}6992685041988275297132819001\allowbreak{}2775984382618962591254350411\allowbreak{}6482991796281053413109329555\allowbreak{}1994079490694698901661705042\allowbreak{}3688740095207071796191852083\allowbreak{}91844000$\\
0 & $-1494343891278352595771732503\allowbreak{}2158028580269247602660888000\allowbreak{}1075297479834391197495798987\allowbreak{}3072819954966664193493367937\allowbreak{}6610653297147630693112233900\allowbreak{}1377228843799294547688571033\allowbreak{}3041392281907380251784632302\allowbreak{}1935065621455454259544768227\allowbreak{}9060225697860047656569793158\allowbreak{}7394610064716731065038954142\allowbreak{}2564447887893333029360487049\allowbreak{}4893557245780920240673948959\allowbreak{}5714068341995599361640000$\\
\bottomrule
\end{longtable}

%% file: main.bbl
\begin{bibdiv}
\begin{biblist}

\bib{Magma}{article}{
      author={Bosma, Wieb},
      author={Cannon, John},
      author={Playoust, Catherine},
       title={The {M}agma algebra system. {I}. the user language},
        date={1997},
     journal={Journal of Symbolic Computation},
      volume={24},
      number={3--4},
       pages={235\ndash 265},
}

\bib{BaileyFried2002}{incollection}{
      author={Bailey, Paul},
      author={Fried, Michael~D.},
       title={{H}urwitz monodromy, spin separation and higher levels of a
  modular tower},
        date={2002},
   booktitle={Arithmetic fundamental groups and noncommutative algebra},
      editor={Fried, Michael~D.},
      editor={Ihara, Yasutaka},
      series={Proceedings of Symposia in Pure Mathematics},
      volume={70},
   publisher={American Mathematical Society},
     address={Providence, RI},
       pages={79\ndash 220},
         url={https://arxiv.org/abs/math/0104289},
      review={\MR{1935406}},
}

\bib{BrumerKramer2012}{article}{
      author={Brumer, Armand},
      author={Kramer, Kenneth},
       title={Arithmetic of division fields},
        date={2012},
     journal={Proceedings of the American Mathematical Society},
      volume={140},
      number={9},
       pages={2981\ndash 2995},
}

\bib{Boege1990}{article}{
      author={B{\"o}ge, Sigrid},
       title={{Witt-Invariante und ein gewisses Einbettungsproblem}},
        date={1990},
     journal={Journal f\"ur die reine und angewandte Mathematik},
      volume={410},
       pages={153\ndash 159},
}

\bib{BasuPollackRoy2006}{book}{
      author={Basu, Saugata},
      author={Pollack, Richard},
      author={Roy, Marie-Fran{\c c}oise},
       title={Algorithms in real algebraic geometry},
     edition={Second},
      series={Algorithms and Computation in Mathematics},
   publisher={Springer-Verlag},
     address={Berlin},
        date={2006},
      volume={10},
}

\bib{Butler1993}{article}{
      author={Butler, Gregory},
       title={The transitive groups of degree fourteen and fifteen},
        date={1993},
     journal={Journal of Symbolic Computation},
      volume={16},
      number={5},
       pages={413\ndash 422},
}

\bib{Atlas1985}{book}{
      author={Conway, John~H.},
      author={Curtis, Robert~T.},
      author={Norton, Simon~P.},
      author={Parker, Richard~A.},
      author={Wilson, Robert~A.},
       title={Atlas of finite groups: Maximal subgroups and ordinary characters
  for simple groups},
   publisher={Clarendon Press},
     address={Oxford},
        date={1985},
}

\bib{Crespo1997}{article}{
      author={Crespo, Teresa},
       title={{G}alois representations, embedding problems and modular forms},
        date={1997},
     journal={Collectanea Mathematica},
      volume={48},
      number={1--2},
       pages={63\ndash 83},
}

\bib{SGA1}{book}{
      author={Grothendieck, Alexander},
       title={Rev\^etements \'etales et groupe fondamental ({SGA} 1)},
      series={Lecture Notes in Mathematics},
   publisher={Springer-Verlag},
     address={Berlin--New York},
        date={1971},
      volume={224},
        note={Augment\'e de deux expos\'es de M. Raynaud},
}

\bib{TransGrp}{manual}{
      author={Hulpke, Alexander},
       title={{TransGrp}: Transitive groups library},
        date={2023},
         url={https://www.math.colostate.edu/~hulpke/transgrp},
        note={Version 3.6.5},
}

\bib{Klueners2000}{article}{
      author={Kl{\"u}ners, J{\"u}rgen},
       title={A polynomial with {G}alois group {$\mathrm{SL}_2(11)$}},
        date={2000},
     journal={Journal of Symbolic Computation},
      volume={30},
      number={6},
       pages={733\ndash 737},
}

\bib{Koenig2017}{article}{
      author={K{\"o}nig, Joachim},
       title={Computation of {H}urwitz spaces and new explicit polynomials for
  almost simple {G}alois groups},
        date={2017},
     journal={Mathematics of Computation},
      volume={86},
      number={305},
       pages={1473\ndash 1498},
         url={https://arxiv.org/abs/1512.05533},
}

\bib{Mestre1990}{article}{
      author={Mestre, Jean-Fran{\c c}ois},
       title={Extensions r\'eguli\`eres de {$\mathbb Q(T)$} de groupe de
  {G}alois {$\widetilde A_n$}},
        date={1990},
     journal={Journal of Algebra},
      volume={131},
      number={2},
       pages={483\ndash 495},
}

\bib{Mestre1991}{incollection}{
      author={Mestre, Jean-Fran{\c c}ois},
       title={Construction de courbes de genre 2 \`a partir de leurs modules},
        date={1991},
   booktitle={Effective methods in algebraic geometry},
      series={Progress in Mathematics},
      volume={94},
   publisher={Birkh\"auser},
     address={Boston, MA},
       pages={313\ndash 334},
}

\bib{Mestre1994b}{article}{
      author={Mestre, Jean-Fran{\c c}ois},
       title={Annulation, par changement de variable, d'\'el\'ements de
  {$\mathrm{Br}_2(k(x))$} ayant quatre p\^oles},
        date={1994},
     journal={Comptes rendus de l'Acad\'emie des sciences. S\'erie I,
  Math\'ematique},
      volume={319},
      number={5},
       pages={529\ndash 532},
}

\bib{Mestre1994}{article}{
      author={Mestre, Jean-Fran{\c c}ois},
       title={Construction d'extensions r\'eguli\`eres de {$\mathbb Q(T)$} \`a
  groupes de {G}alois {$\mathrm{SL}_2(\mathbb F_7)$} et {$\widetilde M_{12}$}},
        date={1994},
     journal={Comptes rendus de l'Acad\'emie des sciences. S\'erie I,
  Math\'ematique},
      volume={319},
      number={8},
       pages={781\ndash 782},
}

\bib{Plans2007}{article}{
      author={Plans, Bernat},
       title={Generic {G}alois extensions for {$\mathrm{SL}_2(\mathbb F_5)$}
  over {$\mathbb Q$}},
        date={2007},
     journal={Mathematical Research Letters},
      volume={14},
      number={3},
       pages={443\ndash 452},
}

\bib{Ramakrishna2002}{article}{
      author={Ramakrishna, Ravi},
       title={Deforming {G}alois representations and the conjectures of {S}erre
  and {F}ontaine--{M}azur},
        date={2002},
     journal={Annals of Mathematics (2)},
      volume={156},
      number={1},
       pages={115\ndash 154},
}

\bib{Ramakrishna1999}{article}{
      author={Ramakrishna, Ravi},
       title={Lifting {G}alois representations},
        date={1999},
     journal={Inventiones Mathematicae},
      volume={138},
      number={3},
       pages={537\ndash 562},
}

\bib{Serre2008}{book}{
      author={Serre, Jean-Pierre},
       title={Topics in {G}alois theory},
     edition={Second},
      series={Research Notes in Mathematics},
   publisher={A K Peters},
     address={Wellesley, MA},
        date={2008},
      volume={1},
}

\bib{Serre1968}{book}{
      author={Serre, Jean-Pierre},
       title={Abelian $l$-adic representations and elliptic curves},
   publisher={W. A. Benjamin},
     address={New York--Amsterdam},
        date={1968},
}

\bib{Serre1984}{article}{
      author={Serre, Jean-Pierre},
       title={L'invariant de {W}itt de la forme {$\operatorname{Tr}(x^2)$}},
        date={1984},
     journal={Commentarii Mathematici Helvetici},
      volume={59},
      number={4},
       pages={651\ndash 676},
         url={https://eudml.org/doc/139995},
}

\bib{Sonn1980}{article}{
      author={Sonn, Jack},
       title={{$\mathrm{SL}(2,5)$} and {F}robenius {G}alois groups over
  {$\mathbb Q$}},
        date={1980},
     journal={Canadian Journal of Mathematics},
      volume={32},
      number={2},
       pages={281\ndash 293},
}

\bib{PARI2}{manual}{
      author={{The PARI Group}},
       title={{PARI/GP}, version 2.17.4},
     address={Universit\'e de Bordeaux},
        date={2025},
         url={https://pari.math.u-bordeaux.fr/},
}

\bib{GAP}{manual}{
      author={{The GAP Group}},
       title={{GAP}---groups, algorithms, and programming, version 4.16.0},
        date={2026},
         url={https://www.gap-system.org},
        note={Official release 4.16.0; the certificate transcript records and
  checks the runtime version and the loaded TransGrp version},
}

\bib{SageMath}{manual}{
      author={{The Sage Developers}},
       title={{SageMath}, the {S}age mathematics software system, version
  10.9},
        date={2026},
         url={https://www.sagemath.org},
        note={Version 10.9},
}

\bib{Uttenthal2025}{misc}{
      author={Uttenthal, Peter~Vang},
       title={Level-raising of even representations of tetrahedral type and
  equidistribution of lines in the projective plane},
        date={2025},
        note={arXiv:2310.14352, version 2, 22 December 2025},
}

\bib{YangModularEquations}{misc}{
      author={Yang, Yifan},
       title={Computing modular equations for {Shimura} curves},
        date={2012},
        note={arXiv:1205.5217},
}

\bib{YangSchwarzian}{article}{
      author={Yang, Yifan},
       title={Schwarzian differential equations and {Hecke} eigenforms on
  {Shimura} curves},
        date={2013},
     journal={Compositio Mathematica},
      volume={149},
      number={1},
       pages={1\ndash 31},
         url={https://arxiv.org/abs/1110.6284},
}

\bib{Zywina2015}{article}{
      author={Zywina, David},
       title={The inverse {G}alois problem for {$\mathrm{PSL}_2(\mathbb
  F_p)$}},
        date={2015},
     journal={Duke Mathematical Journal},
      volume={164},
      number={12},
       pages={2253\ndash 2292},
         url={https://arxiv.org/abs/1303.3646},
}

\end{biblist}
\end{bibdiv}
